\documentclass[reqno,11pt]{amsart}
\usepackage{amsmath,amsthm,amsfonts,amssymb,mathrsfs,bm,graphicx,stmaryrd}
\usepackage{mathtools}
\usepackage[usenames]{xcolor}
\usepackage[colorlinks=true,linkcolor=blue]{hyperref}
\usepackage[letterpaper,hmargin=1.0in,vmargin=1.0in]{geometry}
\parskip 	\smallskipamount

\newtheorem{theorem}{Theorem}[section]
\newtheorem{lemma}[theorem]{Lemma}

\newtheorem{proposition}[theorem]{Proposition}

\newtheorem{remark}[theorem]{Remark}

\newtheorem{maintheorem}{Theorem}

\def\N{\mathbb{N}}

\def\P{\mathbb{P}}
\def\Z{\mathbb{Z}}
\def\R{\mathbb{R}}

\def\E{\mathbb{E}}

\newcommand{\cf}{\mathcal{F}}
\newcommand{\cA}{\mathcal{A}}
\newcommand{\cB}{\mathcal{B}}

\newcommand{\cT}{\mathcal{T}}

\newcommand{\ce}{\mathcal{E}}
\newcommand{\cL}{\mathcal{L}}
\newcommand{\cg}{\mathcal{G}}
\newcommand{\cR}{\mathcal{R}}
\newcommand{\cJ}{\mathcal{J}}

\newcommand{\Addresses}{{
  \bigskip
  \footnotesize

  Riddhipratim Basu, \textsc{International Centre for Theoretical Sciences, Tata Institute of Fundamental Research, Bangalore, India}\par\nopagebreak
  \textit{E-mail address}: \texttt{rbasu@icts.res.in}

  \medskip

  Christopher Hoffman, \textsc{Department of Mathematics, University of Washington, Seattle, WA, USA}\par\nopagebreak
  \textit{E-mail address}: \texttt{hoffman@math.washington.edu}

  \medskip

  Allan Sly, \textsc{Department of Mathematics, Princeton University, Princeton, NJ, USA}\par\nopagebreak
  \textit{E-mail address}: \texttt{allansly@princeton.edu}

}}

\begin{document}
\title[Nonexistence of Bigeodesics]{Nonexistence of Bigeodesics in Integrable Models of Last Passage Percolation}

\author{Riddhipratim Basu
 \and
Christopher Hoffman
\and
Allan Sly
}

\date{}
\maketitle

\begin{abstract}
Bi-infinite geodesics are fundamental objects of interest in planar first passage percolation. A longstanding conjecture states that under mild conditions there are almost surely no bigeodesics, however the result has not been proved in any case. For the exactly solvable model of directed last passage percolation on $\Z^2$ with i.i.d.\ exponential passage times, we study the corresponding question and show that almost surely the only bigeodesics are the trivial ones, i.e., the horizontal and vertical lines. The proof makes use of estimates for last passage time available from the integrable probability literature to study coalescence structure of finite geodesics, thereby making rigorous a heuristic argument due to Newman \cite{ADH15}.
\end{abstract}


\section{Introduction}

We consider the following directed last passage percolation (LPP) model on $\Z^2$. Let $\{\xi_{v}\}_{v\in \Z^2}$ be an i.i.d.\ collection of $\mbox{Exp}(1)$ random variables and let us associate weight $\xi_{v}$ to each vertex $v\in \Z^2$. Define $u\preceq v$ if $u$ is co-ordinate wise smaller or equal than $v$ in $\Z^2$. For any oriented path $\gamma$ from $u$ to $v$ let the passage time of $\gamma$ be defined by
$$\ell(\gamma):=\sum_{v'\in \gamma\setminus \{v\}} \xi_{v'}.$$
For $u\preceq v$ define the last passage time from $u$ to $v$, denoted $T_{u,v}$ by $T_{u,v}:=\max_{\gamma} \ell (\gamma)$ where the maximum is taken over all up/right oriented paths from $u$ to $v$. Observe that by continuity of the exponential distribution, almost surely there exists a unique  path between every pair of (ordered) points $u$ and $v$ that attains this maximum. We shall denote by $\Gamma_{u,v}$ the path between $u$ and $v$ that attains the last passage time $T_{u,v}$ and call $\Gamma_{u,v}$ the geodesic between $u$ and $v$. 

Our object of interest is a bigeodesic, a bi-infinite up/right path $\gamma= \{v_i\}_{i\in \Z}$ such that for each  $i<j$ the restriction of $\gamma$ between $v_i$ and $v_j$ is the geodesic from $v_i$ to $v_j$. It is trivial to observe that the  horizontal and vertical lines, that is, the lines $\{x=i\}$ and $\{y=j\}$ for $i,j\in \Z$, are bigeodesics. We call these bigeodesics the trivial bigeodesics and any other bigeodesic a non-trivial bigeodesic. Our main theorem in this paper is the following.

\begin{maintheorem}
\label{t:bi}
For directed last passage percolation on $\Z^2$ with i.i.d.\ exponential passage times, almost surely there does not exist any non-trivial bigeodesic.
\end{maintheorem}

\subsection{Background}
Kardar, Parisi, and Zhang predicted in their seminal work \cite{KPZ86} that a large class of randomly growing interfaces exhibit a universal behaviour that is now known as KPZ universality, including the longitudinal and transversal fluctuation exponents of $1/3$ and $2/3$ respectively. Directed last passage percolation and first passage percolation\footnote{Planar first passage percolation models are defined by putting i.i.d.\ non-negative weights on the edges of $\Z^2$ and setting the first passage time between two vertices to be equal to the weight of the minimum weight path (geodesic) between two vertices.} models on the plane are believed to belong to the KPZ universality class under very general conditions on the passage time distributions. However, the scaling exponents and the scaling limits have only been rigorously established for a handful of so-called integrable models (e.g.\ LPP with exponential or geometric weights) where exact distributional formulae for the passage times are available due to some remarkable bijections and connections with random matrix theory and orthogonal polynomials. Although some geometric consequences of the algebraic formulae were already studied by Johansson \cite{J00} who established the scaling exponent $2/3$ for transversal fluctuation in Poissonian LPP, sharper geometric estimates and interesting consequences thereof has only recently started being explored \cite{BSS14, BSS17++, BGH17,BG18}.

In another related but separate direction of works, a lot of progress has been made in studying planar first passage percolation, another model believed to be in the KPZ universality class that, however, is not known to be exactly solvable. In the absence of exact formulae, the study of first passage percolation has relied mostly on a geometric understanding of the geodesics, the study of which was initiated by Newman and co-authors as summarized in his ICM paper~\cite{New95} where certain coalescence results are established under curvature assumptions on the limit shape. Although much less is rigorously known, the connection between understanding properties of semi-infinite and bi-infinite geodesics, limit shapes and the KPZ predicted fluctuation exponents has been clear for some years.  Much progress has been made in recent years in understanding the geodesics starting with the idea of Hoffman~\cite{H08} of studying infinite geodesics using Busemann functions. These techniques have turned out to be extremely useful, providing a great deal of geometric information on the structure of geodesics in first passage percolation~\cite{DH14,DH17,AH16}. These techniques have also recently been applied to last passage percolation models with or without integrable structure \cite{GRS15,GRS15+, Pim16}.

The question of the existence of bigeodesics in planar first passage percolation has been one of the most important longstanding problems in the field. This question appeared in \cite{Kes86} where it is attributed to Furstenberg. See \cite[Section 5]{ADH15} for a more detailed account of the history of this problem and its connection to non-trivial ground state in Ising models with i.i.d.\ coupling. Although Benjamini and Tessera recently showed that bigeodesics do exist for first passage percolation on certain hyperbolic graphs \cite{BT17}, it is believed that under some mild conditions on the passage time distribution almost surely bigeodesics do not exist (observe that there are no trivial bigeodesics in the first passage percolation setting) for the two dimensional Euclidean lattice. However, it is only rigorously known that under certain regularity assumptions on the boundary of limit shape bigeodesics along fixed directions do not exist  \cite{DH14,DH17,AH16}. In this paper, we prove the nonexistence of bigeodesics for the exactly solvable model of exponential LPP, where, in addition to the limit shape being explicitly known, more information about the coalescence structure of finite geodesics can be obtained from the moderate deviation estimates available in the integrable probability literature.

\subsection{An outline of the Argument}
In an AIM workshop in 2015, Newman presented a heuristic argument for almost sure non-existence of bigeodesics in FPP predicated on the transversal fluctuation exponent $\xi > 1/2$; see \cite{ADH15}. Part of this paper follows the general outline of that argument with, however, some significant modifications and additional ingredients.
To implement this program we establish new results about the coalescence structure of geodesics in exponential last passage percolation, which are of independent interest and useful in other contexts as well~(see e.g.\ \cite{BSS17++}).

First observe the following: by translation invariance and ergodicity, we know that existence of a non-trivial bigeodesic is a $0-1$ event and hence it follows that if  almost surely non-trivial bigeodesics exist, then with positive probability there must exist non-trivial bigeodesics passing through the origin, denoted $\mathbf{0}$ (more generally for $r\in \Z$, we shall use $\mathbf{r}$ to denote the point $(r,r)$). We shall prove Theorem~\ref{t:bi} by showing that almost surely there does not exist any non-trivial bigeodesic passing through~$\mathbf{0}$. Let $\gamma=\{v_i\}_{i\in \Z}$ be a bi-infinite path passing through $\mathbf{0}$. Without loss of generality assume $v_0=\mathbf{0}$. Let us set $v_i:=(x_i,y_i)$. Observe that if $\gamma$ is a bigeodesic then $\gamma_{+}:=\{v_0,v_1,\ldots\}$ and $\gamma_{-}:=\{v_0,v_{-1},\ldots\}$ are both semi-infinite geodesic rays \footnote{Semi-infinite geodesics, or geodesic rays, are naturally defined as follows. A path $\gamma=\{v_i\}_{i\in \Z_{\geq 0}}$ is called a semi-infinite geodesic if $v_{i}\preceq v_{i+1}$ for all $i$ (or $v_{i+1}\preceq v_{i}$ for all $i$), and the restriction of $\gamma$ between $v_{i}$ and $v_{j}$ is a geodesic from $v_{i}$ to $v_{j}$ for all $i<j$ (resp.\ for all $i>j$).}. It is known \cite{New95,FP05} that almost surely every geodesic ray emanating from a fixed vertex has a direction, i.e., except on a set of zero probability $\lim_{i\to \infty} \frac{y_i}{x_i}:=h(\gamma_{+})\in [0,\infty]$ and $\lim_{i\to -\infty} \frac{y_i}{x_i}:=h(\gamma_{-})\in [0,\infty]$ exist \footnote{{Our arguments can be used to show that almost surely there does not exist a bigeodesic $\gamma$ with $h(\gamma_{+})\neq h(\gamma_{-})$ but this will not be of particular use to us.}}. For a bigeodesic $\gamma$ passing though $\mathbf{0}$ we shall call $h(\gamma_{+})$ and $h(\gamma_{-})$  the \textbf{forward limiting direction} and the \textbf{backward limiting direction} of $\gamma$ respectively. As already pointed out, the vertical and horizontal directions are somewhat special, we shall take care of them separately. For $h\in (0,1)$, let $\ce_{h}$ denote the event that there exists a bigeodesic passing though $\mathbf{0}$ such that either its forward limiting direction is in $(h,\frac{1}{h})$, or its backward limiting direction is in $(h,\frac{1}{h})$ (observe that such a geodesic must be non-trivial). Let $\ce_{*}$ denote the event that there exists a {non-trivial} bigeodesic $\gamma$ passing through the origin which has either $h(\gamma_{-})=h(\gamma_{+})=0$ (i.e., it is horizontally directed) or $h(\gamma_{-})=h(\gamma_{+})=\infty$ (i.e., it is vertically directed). It is immediate that Theorem \ref{t:bi} will follow from the next two propositions ({notice that the case that exactly one of $h(\gamma_{-})$ and $h(\gamma_{+})$ is in $\{0,\infty\}$ is covered by $\ce_{h}$ for $h$ sufficiently small}).

\begin{proposition}
\label{p:h}
For each $h\in (0,1)$, we have $\P(\ce_{h})=0$.
\end{proposition}

\begin{proposition}
\label{p:axial}
We have $\P(\ce_{*})=0$.
\end{proposition}

{Observe that the axial directions are special in LPP (in contrast to FPP) since the model is directed and trivial geodesics do exist in axial directions. Therefore we shall need a separate argument ruling out non-trivial bigeodesics which are axially directed.}  Let us, for now, focus on the situation of Proposition \ref{p:h}, and describe how this proposition is established following Newman's general heuristics for the bigeodesics problem in the FPP setting. Clearly it suffices to prove Proposition \ref{p:h} for $h$ sufficiently small. Let $S_{n}$ denote the square $[-n,n]^2 \cap \Z^2$. We shall denote the union of its left and bottom side by $\mathsf{Ent}_n$ and the union of its top and right side by $\mathsf{Exit}_{n}$. Observe that any bi-infinite path through $\mathbf{0}$ must enter $S_n$ through a point on $\mathsf{Ent}_n$, and exit $S_{n}$ via a point on $\mathsf{Exit}_{n}$. Clearly if the path is a bigeodesic, then its restriction to $S_n$ must give a geodesic between a point on $\mathsf{Ent}_n$ and a point on $\mathsf{Exit}_{n}$. Moreover, on $\ce_{h}$, one must also have that for all $n$ sufficiently large, the line joining the endpoints of the putative bigeodesic restricted to $S_n$ must have slope in  $(\frac{h}{2}, \frac{2}{h})$. Let $\ce_{n,h}$ denote the event that there exists points $u\in \mathsf{Ent}_n$ and $w\in \mathsf{Exit}_n$ such that $\mbox{slope}(u,w)\in (\frac{h}{2}, \frac{2}{h})$ (for $u\preceq w\in \Z^2$, $\mbox{slope}(u,w)$ shall denote the slope of the straight line joining $u$ and $w$) and $\mathbf{0}\in \Gamma_{u,w}$. Clearly if $\P(\ce_{h})>0$ then $\liminf_{n\to\infty} \P(\ce_{n,h})>0$. This is contradicted by the following proposition which, therefore, implies Proposition \ref{p:h}.

\begin{proposition}
\label{p:o1}
Let $h\in (0,1)$ be fixed. There exists $C=C(h)>0$ such that $\P(\ce_{n,h})\leq Cn^{-1/3}$ for infinitely many $n$.
\end{proposition}

Newman's heuristic for showing that $\P(\ce_{n,h})=o(1)$ is the following. Divide the intervals $\mathsf{Ent}_n$ and $\mathsf{Exit}_n$ into disjoint subintervals of length $n^{\xi}$ where $\xi$ is the transversal fluctuation exponent (known to be equal to $2/3$ in our case). For most pairs of  intervals $(I,J)$, the point $\mathbf{0}$ is ``far" (at the transversal fluctuation scale) from the straight lines joining points in $I$ to points in $J$, so the contribution for such pairs  should be negligible and the main contribution should come from the ``opposite pairs". Also for each pair of ``opposite" sub-intervals $I$ and $J$, $I\subseteq \mathsf{Ent}_n$, and $J\subseteq \mathsf{Exit}_n$, the geodesics from points in $I$ to points in $J$ ``should coalesce" and hence the chance of there being any geodesic passing through the origin should be $\approx n^{-\xi}$. Taking a union bound over $(n^{1-\xi})$ many pairs of opposite intervals, we should get the required probability bound as long as $\xi >1/2$.

There are a number of obvious issues with this heuristic, even if the transversal fluctuation exponent in known to be bigger than $1/2$, as was already pointed out in \cite{ADH15}. First, as was shown recently in \cite{Pim16, BSS17++} coalescence (of all geodesics) in an on-scale rectangle (i.e., an $n\times n^{2/3}$ rectangle) happens with positive probability, but not with high probability. Second, one needs to deal with the correlated events of coalescence and the geodesic passing through the origin. To circumvent these issues we show that even though all paths might not coalesce, most of the paths do (see Theorem \ref{t:multi2}). The other issue  is to deal with the contribution of the pairs of intervals that are not exactly opposite one another. This issue is circumvented by an averaging argument, where instead of looking at the probability of some geodesic passing through the origin we look at the average number of vertices near the origin that are on such geodesics.

{More precisely, we show in Lemma \ref{l:sidetoside} (a refinement of Theorem \ref{t:multi2}) that the for intervals $I\subset \mathsf{Ent}_{n}$ and $J\subset \mathsf{Exit}_{n}$, each of length $n^{2/3}$ such that the slope between the midpoints of $I$ and $J$ is not too extreme, the expected number of vertices in $[-\frac{nh}{100}, \frac{nh}{100}]^2 \cap \Z^2$ that lie on some geodesic from $I$ to $J$ is linear in $n$. Summing over $O(n^{2/3})$ pairs of $(I,J)$ gives that the expected number of vertices  in $[-\frac{nh}{100}, \frac{nh}{100}]^2 \cap \Z^2$ that lie on some geodesic between $\mathsf{Ent}_{n}$ to $\mathsf{Exit}_{n}$ such that the slope between its endpoints is not too extreme is $O(n^{5/3})$. Therefore the average probability that a vertex in $[-\frac{nh}{100}, \frac{nh}{100}]^2 \cap \Z^2$ lies on such a geodesic is $O(n^{-1/3})$. Arguing using translation invariance that the same estimate is true for the origin if $n$ is replaced by $n'$ for some $n'\in [n,n+\frac{nh}{100}]$ then completes the proof of Proposition \ref{p:o1}.}

\subsection*{Disjoint Geodesics}
En route to the proof of the coalescence of most geodesics (Theorem \ref{t:multi2}) alluded to above, we first establish a result of independent interest  concerning the number of disjoint geodesics across an on-scale rectangle, i.e., a rectangle of size $n\times n^{2/3}$ whose sides of length $n$ are parallel to the line $x=y$ (time direction) and the sides of length $n^{2/3}$ is parallel to the line $x+y=0$ (space direction).  The result says essentially says that the maximum number of disjoint geodesics from one side of the rectangle to the other is tight at $O(1)$ scale, and has nice stretched exponential tails. Rarity of disjoint geodesics is a question of independent interest, and has been investigated in \cite{H17a} in the context of Brownian last passage percolation using the Brownian Gibbs property of \cite{CH14}. Showing that a large number of disjoint geodesics is sufficiently rare has a number of applications. In this paper, we shall use this result to prove Proposition \ref{p:o1}. In \cite{BSS17++}, this is used to prove optimal tail estimates for distance to coalescence for semi-infinite geodesics started at distinct points. For applications in studying the locally Brownian nature of Airy processes, see \cite{H16, H17a, H17b, H17c}.

For $r\in \Z$, let $\mathbb{L}_{r}$ denote the line $x+y=2r$. Let $A_{n}$ (resp.\ $B_{n}$) denote the line segment on $\mathbb{L}_{0}$ (resp.\ $\mathbb{L}_{n}$) of length $2n^{2/3}$ with midpoint $\mathbf{0}$ (resp.\ $\mathbf{n}$). For points $u,v$ on $A_n$ (or on $B_{n}$) we say $u<v$ if $v=u+i(-1,1)$ for some $i\in \N$. For $\ell\in \N$, let $\widetilde{\ce}_{\ell}$ denote the event 
that there exists $u_1< u_2 < \cdots < u_{\ell}$ on $A_{n}$, and $v_1< v_2 < \cdots < v_{\ell}$ on $B_{n}$, such that the geodesics  $\Gamma_{u_i,v_i}$ are disjoint. The next theorem is the second main result of this paper.

\begin{maintheorem}
\label{t:rare}
There exists constant $n_0,\ell_0\in \N$ such that for all $n>n_{0}$ and for all $\ell _0 <\ell < n^{0.01}$ we have
$$ \P(\widetilde{\ce}_{\ell})\leq e^{-c\ell^{1/4}}$$
for some absolute constant $c>0$.
\end{maintheorem}

Observe that Theorem \ref{t:rare} immediately implies that if $N_{n}$ denotes the maximum number of pairwise disjoint geodesics from points on $A_n$ to points on $B_{n}$, then we have $\E N_{n}\leq C$ for some absolute constant $C$. A variant of this can be used to obtain the optimal upper bound for the so-called midpoint problem; see Remark \ref{r:midpoint}.

\subsection*{The axial directions}
Before wrapping up this section, let us present a brief outline of the argument for proving Proposition \ref{p:axial}. Let us only consider the vertical direction. Simple translation invariance and ergodicity considerations show that there cannot exist a vertically directed bigeodesic which only moves finitely many steps in the horizontal direction. {Indeed, for $a<b\in \Z$, clearly almost surely there can be at most one such geodesic that moves from the line $x=a$ to the line $x=b$. By translation invariance, it follows that the location of the first jump of such a geodesic to the right of the line $x=a$ is invariant under vertical translations which leads to a contradiction (see Lemma \ref{l:fw})}. So it suffices to show that there cannot exist any semi-infinite geodesic started from the origin directed vertically upwards that moves infinitely many steps to the right. We prove this by contradiction. If such a geodesic exists with positive probability, then with positive probability it will also take $M$ rightward steps before $L$ upward steps for some large $M$ and large $L$ depending on $M$. 

To rule this out, we establish the following two results. First we show that for $\varepsilon$ arbitrarily small the transversal fluctuation of the geodesic 
from $\mathbf{0}$ to $(\varepsilon n, n)$ is $O(\varepsilon^{2/3}n^{2/3})$ with high probability (see Proposition \ref{p:tfbasic}); this generalizes Johansson's transversal fluctuation result \cite{J00} to steep geodesics. We further prove a local version of the above transversal fluctuation result showing that the local transversal fluctuation of the geodesic from $\mathbf{0}$ to $(\varepsilon n, n)$ at height $L$ is $O(\varepsilon^{2/3}L^{2/3})$ (see Theorem \ref{l:ltf}). This generalizes Theorem 3 of \cite{BSS17++}, where a similar result was proved for $\varepsilon$ bounded away from $0$ and $\infty$.

Once we have these results at our disposal we can simply take $\varepsilon$ sufficiently small depending on $L$, and argue that if the geodesic from $\mathbf{0}$ to $(\varepsilon n, n)$ took $M$ rightward steps before $L$ upward steps then it would have atypically large transversal fluctuation at height $L$. Observing that any semi-infinite geodesic started at $\mathbf{0}$ and directed vertically upward will be to the left of the geodesic $\mathbf{0}$ to $(\varepsilon n, n)$ for all $n$ sufficiently large completes the proof. 


\subsection*{Notes on Subsequent Results}
{In the two years since this paper was completed and posted on arXiv (in November 2018) there has been a number of results in related problems using both the ideas in this article as well as independent methods. In September 2019, the paper \cite{BBS19} was posted where the authors provided a different proof of Theorem \ref{t:bi} using a stationary LPP and exact formulae for joint distribution of Busemann increments. In \cite{Ale20}, the methods of this paper where adapted and extended to treat the case of first passage percolation in general dimension under the unverified assumptions of curvature of limit shape and one point moderate deviation estimates at the standard deviation scale (these are analogues of results that are known in the exponential LPP case, see Theorem \ref{t:onepoint} below). Finally, in \cite{BGHH20}, a variant of Theorem \ref{t:rare} was derived under a generic class of assumptions (curvature of limit shape and one point moderate deviation estimates being the primary ones) on the underlying LPP models en route studying the $k$-geodesic watermelon, the maximal collection of $k$ disjoint paths with maximal total weight.}

\subsection*{Other Integrable Models}
{We worked with the specific integrable model of planar exponential LPP in this paper; however, it will be clear to the reader that we have used little information specific to the exponential model. We believe that our methods can be adapted for other integrable models of last passage percolation as well; this is reflected in our choice of the title for this paper. Indeed, except for the weak convergence for passage times to the GUE Tracy-Widom distribution, the primary ingredient we need from the model is the one point moderate deviation estimates for the last passage times (Theorem \ref{t:onepoint}) and its consequences on tails of passage times across an on scale parallelogram (Theorem \ref{t:supinf}). Analogues of Theorem \ref{t:onepoint} exist in the literature for all known planar models of integrable LPP (see \cite{LM01,LMS02} for Poissonian LPP, \cite{BXX01,Jo99} for geometric LPP, and \cite{LR09} for Brownian LPP). The proof of Theorem \ref{t:supinf} using Theorem \ref{t:onepoint} can be imitated to obtain analogous results for all known models of integrable last passage percolation (This was done in the context of Poissonian LPP in \cite{BSS14}, see \cite{BGHH20} for similar results in geometric LPP and \cite{GH20} for results of a similar flavour in Brownian LPP). To refrain from relying on unpublished results, and to keep the exposition succinct, we shall not pursue any of these other models in this paper. However, we reiterate that we believe that our arguments proving Theorem \ref{t:bi} can be adapted to prove analogous results for the other three integrable models of planar LPP. Notice that, for Poissonian and geometric LPP, some additional arguments will be needed to deal with the nonuniqueness of geodesics (uniqueness of geodesic is used in some proofs in Section \ref{s:quick}).}  


\subsection*{Organization of the paper}
The rest of the paper is organized as follows. In Section \ref{s:prelim}, we recall the known basic inputs from integrable probability that we use throughout the paper, and state a couple of new results (Theorem \ref{p:tfbasic} and Theorem \ref{l:ltf}) to deal with axially directed bigeodesics. In Section \ref{s:rare} we prove Theorem \ref{t:rare} (and its generalization Proposition \ref{c:rare}) and a useful consequence Theorem \ref{t:multi2}. In Section \ref{s:h}, we complete the proof of Proposition \ref{p:o1} using Theorem \ref{t:multi2}. In Section \ref{s:axial}, we complete the proof of Theorem \ref{t:bi} by establishing Proposition \ref{p:axial} using Theorem \ref{p:tfbasic} and Theorem \ref{l:ltf} and provide the proofs of the latter two results. 

\subsection*{Acknowledgements}
We thank Manjunath Krishnapur for pointing us to the reference \cite{LR09}, Daniel Ahlberg for useful discussions at the early stage of the project, and Yiping Hu for a careful reading of the paper. RB is partially supported by a Ramanujan Fellowship (SB/S2/RJN-097/2017) from the Science and Engineering Research Board, an ICTS-Simons Junior Faculty Fellowship, DAE project no. RTI4001 via ICTS and Infosys Foundation via the Infosys-Chandrasekharan Virtual Centre for Random Geometry of TIFR. CH is supported by a Simons fellowship and NSF Grants DMS-1712701 and DMS-1954059. AS is supported by NSF grant DMS-1855527, a Simons Investigator grant and a MacArthur Fellowship. Part of this research was performed during a visit of RB to the Princeton Mathematics department whose hospitality he gratefully acknowledges.

\section{Inputs from Integrable Probability and their consequences}
\label{s:prelim}
This paper falls within the general program of understanding the geometry of geodesics in exactly solvable models of last passage percolation using inputs from integrable probability initiated in \cite{BSS14} and continued in \cite{BSS17++, BGH17, BG18, BGZ19}. As such, we use the same integrable inputs, and their consequences developed in these papers, primarily \cite{BSS14}. In this section, we collect all such results. Note that the arguments in \cite{BSS14} were written in the set-up of the exactly solvable model of Poissonian LPP although essentially the same arguments go through for exponential LPP. However, for the sake of completeness, and to avoid citing results that are as yet unpublished, we shall mostly quote these inputs from \cite{BGZ19} where the arguments from \cite{BSS14} are reproduced in detail in the set-up of exponential LPP.


\subsection{One point convergence and moderate deviation estimates}
The two fundamental ingredients are the convergence of the rescaled passage time for the Exponential LPP \cite{Jo99} and a moderate deviation estimate for the same \cite{LR09}.\footnote{Strictly speaking, one usually proves such results in the model of Exponential LPP where the weight of the last vertex is also included in the definition of $T$. However for large $n$ this does not make any difference and we shall ignore this issue henceforth.}
First we recall the Tracy-Widom convergence result. 

\begin{theorem}[{\cite[Theorem 1.6]{Jo99}}]
\label{t:tw}
For each $h\in (0,\infty)$ we have 
$$ \dfrac{T_{\mathbf{0}, (n,hn)}- n(1+\sqrt{h})^{2}}{h^{-1/6}n^{1/3}} \stackrel{d}{\rightarrow} F_{TW}$$
as $n\to \infty$ where $F_{TW}$ denotes the GUE Tracy-Widom distribution.
\end{theorem}

We shall have limited use for the scaling limit except for the well-known fact that GUE Tracy-Widom distribution has negative mean ({see \cite[Lemma A.4]{BGHH20} for a proof}). For our purposes the following moderate deviation estimates will be of paramount importance.

\begin{theorem}[{\cite[Theorem 2]{LR09}}]
\label{t:onepoint}
For each $\psi>1$ There exists $C,c>0$ depending on $\psi$ such that for all $m,n,r\geq 1$ with $\psi^{-1}<\frac{m}{n}< \psi$ and all $x>0$ we have the following:
\begin{enumerate}
\item[(i)] $\P(T_{\mathbf{0}, (m,n)}-(\sqrt{m}+\sqrt{n})^{2} \geq xn^{1/3}) \leq Ce^{-c\min\{x^{3/2},xn^{1/3}\}}$.
\item[(ii)] $\P(T_{\mathbf{0}, (m,n)}-(\sqrt{m}+\sqrt{n})^{2} \leq -xn^{1/3}) \leq Ce^{-cx^3}$.
\end{enumerate}
\end{theorem}

Observe that Theorem \ref{t:onepoint} implies that 
\begin{equation}
\label{e:mean}
|\E T_{\mathbf{0}, (m,n)} -(\sqrt{m}+\sqrt{n})^2|\leq C'n^{1/3}
\end{equation}
for some constant $C'$ depending only on $\psi$. To deal with the steep geodesics (see  Section \ref{s:steep} and Section \ref{s:axial}) we shall also need some amount of control on the tails of the passage time when $m/n$ is not bounded away from $0$ and $\infty$. This is provided by \cite[Theorem 2]{LR09} as well. In particular, we have: for $m\geq n\geq 1$ sufficiently large and for all $x>0$
\begin{equation}
\label{e:steep}
\P(T_{\mathbf{0}, (m,n)}-(\sqrt{m}+\sqrt{n})^{2} \geq xm^{1/2}n^{-1/6}) \leq Ce^{-cx}.
\end{equation}
See the discussion after \cite[(3.17)]{LR09} for the case $x> m^{-1/2}n^{1/6}(\sqrt{m}+\sqrt{n})^2)$. For $x\in (0,  m^{-1/2}n^{1/6}(\sqrt{m}+\sqrt{n})^2))$, we also get from \cite[Theorem 2]{LR09} that
\begin{equation}
\label{e:steep2}
\P(T_{\mathbf{0}, (m,n)}-(\sqrt{m}+\sqrt{n})^{2} \leq -xm^{1/2}n^{-1/6}) \leq Ce^{-cx^2}.
\end{equation}
Notice that these together imply 
\begin{equation}
\label{e:meansteep}
|\E T_{\mathbf{0}, (m,n)} -(\sqrt{m}+\sqrt{n})^2|\leq Cm^{1/2}n^{-1/6}
\end{equation}
for all $m\geq n\geq 1$ for some $C>0$.
%

\subsection{Passage times across parallelograms}
One of the most useful consequences of Theorem \ref{t:onepoint} for our purposes will be a similar bound on minimum and maximum passage times across rectangles and parallelograms of dimension $n\times n^{2/3}$. This estimate was developed in \cite{BSS14} for Poissonian LPP; the version we quote here is taken from \cite{BGZ19}. 

Let $U_0$ and $U_{r}$ line segments of length $2r^{2/3}$ on the lines $\{x+y=0\}$ and $\{x+y=2r\}$ respectively with midpoints $(mr^{2/3},-mr^{2/3})$ and $\mathbf{r}$ respectively. The following theorem is the key consequence of Theorem \ref{t:onepoint} for our purposes. 

\begin{theorem}[{\cite[Theorem 4.2]{BGZ19}}]
\label{t:supinf}
For each $\psi<1$, there exists $C,c>0$ depending only on $\psi$ such that for all $|m|<\psi r^{1/3}$ and $U$ as above we have
\begin{enumerate}
\item[(i)] 
for all $x>0$ and $r\ge 1$,
$$\P\left( \inf_{u\in U_0,v\in U_r}  (T_{u,v}-\E T_{u,v}) \leq -xr^{1/3}\right)\leq Ce^{-cx^3}.$$
\item[(ii)]
for all $x>0$ and $r\geq 1$,
$$\P\left( \sup_{u\in U_0,v\in U_r}  (T_{u,v}-\E T_{u,v}) \geq xr^{1/3}\right)\leq Ce^{-c\min\{x^{3/2},xr^{1/3}\}}.$$
\end{enumerate}
\end{theorem}


\subsection{Transversal Fluctuation estimates}
As already mentioned, the transversal fluctuation exponent in exponential LPP is known to be equal to $2/3$, we shall need a quantitative upper bound to that effect. The following result is taken from \cite{BGZ19} which in turn was proved by adapting the arguments appearing in \cite[Theorem 11.1, Corollary 11.7]{BSS14}. 

\begin{theorem}[{\cite[Proposition C.9]{BGZ19}}]
\label{p:tf}
Let $\cA_\phi$ denote the event that the geodesic from $(mr^{2/3},-mr^{2/3})$ to $\mathbf{r}$ exits the strip of width $\phi r^{2/3}$ around the straightline joining the endpoints. For each $\psi<1$, there exist $C,c>0$ such that for all $|m|\leq \psi r^{1/3}$ and $\phi>0$, $r\geq 1$,
$$\P(\cA_{\phi})\leq Ce^{-c\phi^3}.$$
\end{theorem}

\subsection{Estimates for steep geodesics}
\label{s:steep}
To deal with the axial directions we also need to control the transversal fluctuation of geodesics joining two points such the straight line joining them has slope arbitrarily close to $0$ or $\infty$. Unlike the previous results stated in this section, these have has not appeared before in the literature, so we shall provide complete proofs. However, the proofs are quite similar to the results quoted above using \eqref{e:steep}, \eqref{e:steep2} instead of Theorem \ref{t:onepoint}.

We first make the following definitions. For any path $\gamma$ from $\mathbf{0}$ to $(\varepsilon n, n)$, let the local transversal fluctuation of $\gamma$ at length scale $L$ be
$$TF_{L}(\gamma):=\sup\{(x-\varepsilon L)_{+}: (x,L)\in \gamma\}.$$
Our first result controls the transversal fluctuation of geodesics from $\mathbf{0}$ to points near $(\varepsilon n,n)$ at scales close to $\frac{n}{2}$.

\begin{theorem}
\label{p:tfbasic}
There exist constants $C_0,c,x_0, \varepsilon_0>0$ such that for each $\varepsilon\in (\frac{C_0}{n}, \varepsilon_0)$, $m\in (\frac{\varepsilon}{10}, 10 \varepsilon)$, $n$ sufficiently large and $x>x_0$, we have the following. Let $\Gamma$ denote the geodesic from $\mathbf{0}$ to $(mn,n)$. Then we have for each $L\in [\frac{n}{4}, \frac{n}{2}]$,
$$\P(TF_{L}(\Gamma) \geq x\varepsilon^{2/3}n^{2/3})\leq e^{-cx}$$
for some $c>0$. 
\end{theorem}

\begin{remark}
\label{r:tfsteep}
Using Theorem \ref{p:tfbasic} and the chaining argument used in the proof of \cite[Proposition C.9]{BGZ19} one can also upper bound the global transversal fluctuations. i.e., for the geodesic $\Gamma$ from $\mathbf{0}$ to $(\varepsilon n,n)$, one can show that 
$\sup_{0\leq L\leq n} TF_{L}(\Gamma)=O(\varepsilon^{2/3}n^{2/3})$ with large probability as was conjectured in \cite{BM05}. However, we shall not need this result and would refrain from providing the details. This maybe taken up elsewhere. 
\end{remark}

The final result we need will control the local transversal fluctuation of the geodesic from $\mathbf{0}$ to  $(\varepsilon n,n)$ at some large (but $\ll n$) length scale $L$. 

\begin{theorem}
\label{l:ltf}
There exists $\varepsilon_0,x_0,c>0$ and $C_0, L_0, N_0\in \N$ with $C_0<\varepsilon_0L_0$ such that for all $n>N_0$, $L\in (L_0,\frac{n}{8})$, $x>x_0$ and $\frac{C_0}{L}\leq \varepsilon<\varepsilon_0$, we have the following: if $\Gamma$ is the geodesic from $\mathbf{0}$ to $(\varepsilon n, n)$, then
$$\P(TF_{L}(\Gamma) \geq x\varepsilon^{2/3}L^{2/3})\leq e^{-{cx^{1/3}}}.$$
\end{theorem}

Notice that for the case $\varepsilon$ bounded away from $0$, an analogue of Theorem \ref{l:ltf} was proved in \cite[Theorem 3]{BSS17++}. Proofs of Theorem \ref{p:tfbasic} and Theorem \ref{l:ltf} will be provided in Section \ref{s:axial}.

\section{Rarity of Multiple Disjoint Geodesics}
\label{s:rare}

The objective of this section is to prove Theorem \ref{t:rare} and use it to prove Theorem \ref{t:multi2} which will ultimately be applied in the proof of Proposition \ref{p:h}. 

Before delving into the details of the proof of Theorem \ref{t:rare}, let us briefly explain the idea. Recall the basic set up: $A_{n}$ and $B_{n}$ are anti-diagonal line segments of length $2n^{2/3}$ with centres $\mathbf{0}$ and $\mathbf{n}$ respectively, and we want to assert that the maximum number of pairwise disjoint geodesics is typically not too large. First we shall show that the length of the geodesic from any point on $A_{n}$ to any point in $B_{n}$ is unlikely to be too small, i.e., even the minimum geodesic length is typically $4n-\Theta(n^{1/3})$. Now the question is reduced to showing that it is unlikely to have a large number of disjoint paths from $A_n$ to $B_n$ that have length at least $4n-Cn^{1/3}$ for some large $C$ (depending on $\ell$). {Notice that applying the BK inequality directly is not immediately helpful, as for $C$ large the probability of having one geodesic with weight $\geq 4n-Cn^{1/3}$ is rather close to $1$. To circumvent this issue, we shall discretize the set of all paths from $A_{n}$ to $B_{n}$, such that that the typical weight for the best path in any given discretization is much smaller than $4n-Cn^{1/3}$. It turns out that the probability of having $\ell$ many such paths which are atypically large compared to the typical value in the corresponding discretization is sufficiently unlikely so that we can sum over all possible discretizations to get the desired bound in Theorem \ref{t:rare}}.

More specifically we do the following. First, using Theorem \ref{p:tf}, observe that we can restrict to paths completely contained in the rectangle $R_{n}$ whose one pair of opposite sides $U_n$ and $V_n$ are along $A_{n}$ and $B_{n}$ with the same midpoints but have lengths $2\ell^{1/8}n^{2/3}$. For some $s>0$ consider the intersections of the lines $x+y=\frac{in}{s}$ with the rectangle $R_{n}$ (note that the notation in this section is independent of the rest of the paper and the parameter $h$ here used locally has no connection with $h$ used in the statement of Proposition \ref{p:h}). Partition these line segments (of length $2\ell^{1/8}n^{2/3}$) into $\ell^{1/8}t$ segments of equal length \footnote{{Without loss of generality, we shall throughout ignore the rounding issues and omit the floor signs to reduce notational overhead, the reader can easily check that it does not affect any of the arguments in a non-trivial way.}}. We now discretize all possible geodesics from $A_{n}$ to $B_{n}$ (contained in $R_{n}$) according to which of the $t$ segments it intersects for each $i=1,2,\ldots, h$. {If a path is forced to go through a fixed sequence of intervals, it incurs a penalty in its weight due to Theorem \ref{t:tw}, Theorem \ref{t:supinf} and the fact that the mean of the GUE Tracy-Widom distribution is negative. It follows that if $t$ is sufficiently large compared to $s^{2/3}$, the weight of the highest weight path corresponding to a fixed discretization is typically much smaller than a typical geodesic from $A_{n}$ to $B_{n}$; this is the content of Lemma \ref{l:thin1plus}}. Now, we can use the BK inequality to conclude that $\ell$ many such paths with sufficiently high weight existing disjointly is very unlikely. As any set of disjoint paths must be ordered, we can get a good control on the entropy of the size of $\ell$-tuples of possible discretization (counted in Lemma \ref{l:count}), and it turns out that the probability bound coming from the BK inequality is sufficiently small to beat the union bound over all possible $\ell$-tuples of discretizations (see Proposition \ref{p:multi}).

We now move towards making this argument precise. For technical convenience for applications in this paper as well as other applications, we shall prove a more general version of Theorem \ref{t:rare}; there are two ways in which the next result will generalize Theorem \ref{t:rare}. First, instead of considering intervals  $A_n$ and $B_{n}$ centred at points $\mathbf{0}$ and $\mathbf{n}$, the line joining which has slope one, we shall consider the two midpoints with more general slopes. Also, we shall consider intervals whose lengths will be allowed to grow with $\ell$. More specifically, Let $A'_{n}$ (resp.\ $B'_n$) be line segments of length $2\ell^{1/16}n^{2/3}$ for some parallel to the line $x+y=0$ with midpoints $(-mn^{2/3},mn^{2/3})$ and $\mathbf{n}$ respectively. Let $\ce'_{\ell}$ denote the event that there exists $u_1< u_2 < \cdots < u_{\ell}$ on $A'_{n}$, and $v_1< v_2< \cdots  < v_{\ell}$ on $B'_{n}$, such that the geodesics  $\Gamma_{u_i,v_i}$ are disjoint. We have the following result.


\begin{proposition}
\label{c:rare}
For each $\psi<1$, there exists $n_0,\ell_0>0$, such that for all $n>n_0, n^{0.01}>\ell>\ell_0$ and all $m$ with $|m|+\ell^{1/8}<\psi n^{1/3}$ we have 
$\P(\ce'_{\ell})\leq e^{-c\ell^{1/4}}$
for some $c>0$.
\end{proposition}

Clearly Proposition \ref{c:rare} implies Theorem \ref{t:rare}. Hence it suffices to prove only the former, and that is what we move towards now. Proposition \ref{c:rare} is used later to prove Theorem \ref{t:multi2}, and also used in \cite{BSS17++}.
 

Let us fix $n$ be sufficiently large for now and let $\ell < n^{0.01}$ be also fixed and sufficiently large. Let $U_{n}$ and $V_{n}$ be the line segments on $\mathbb{L}_{0}$ and $\mathbb{L}_{n}$ of length $2\ell ^{1/8}n^{2/3}$ with midpoints $(-mn^{2/3},mn^{2/3})$ and $\mathbf{n}$ respectively. Let $R_n$ denote the parallelogram one of whose pairs of opposite sides are $U_{n}$ and $V_{n}$. The following lemma says that geodesics from $A'_{n}$ to $B'_{n}$ will typically be completely contained in $R_{n}$.

\begin{lemma}
\label{l:rec}
Let $\cf_{\ell}$ denote the event that there exist $u\in A'_n$ and $v\in B'_n$ such that $\Gamma_{u,v}$ exits $R_n$. Then for $n,\ell$ and $m$ as in the statement of Proposition \ref{c:rare}, we have $\P(\cf_{\ell})\leq e^{-c\ell^{1/4}}$ for some $c>0$.
\end{lemma}

\begin{proof}
Let $u_0$ and $u'_0$ (resp.\ $v_0$ and $v'_0$) denote the smallest and the largest vertices of $A'_n$ (resp.\ $B'_n$) in the order defined above. It is easy to see that all $\Gamma_{u,v}$'s (for $u\in A'_n, v\in B'_n$) are sandwiched between $\Gamma_{u_0,v_0}$ and $\Gamma_{u'_0,v'_0}$; this fact is often referred to as polymer ordering. So it suffices to upper bound the probability that $\Gamma_{u_0,v_0}$ or $\Gamma_{u'_0,v'_0}$ will exit $R_n$. Noticing that the slope condition in Theorem \ref{p:tf} is satisfied for both $\Gamma_{u_0,v_0}$ and $\Gamma_{u'_0,v'_0}$, the lemma now follows from applying the same.
\end{proof}

\begin{figure}[htbp!]
\centering
\includegraphics[width=0.4\textwidth]{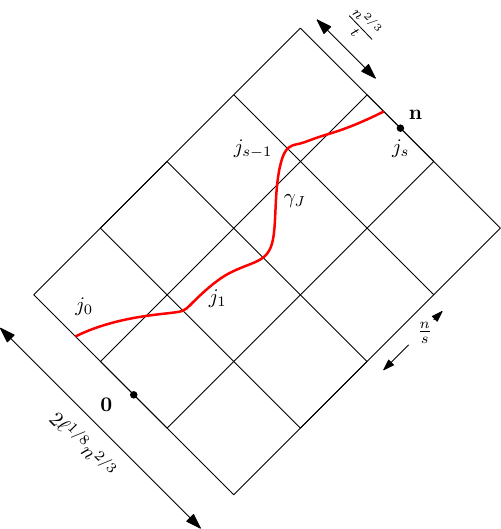}
\caption{The construction for the proof of Proposition \ref{p:multi} in the special case $m=0$: we divide the $n\times 2\ell^{1/8}n^{2/3}$ rectangle $R_{n}$ into an $s\times 2\ell^{1/8}t$ grid of subrectangles of size $(\frac{n}{s})\times (\frac{n^{2/3}}{t})$ where $t$ is chosen to be sufficiently large compared to $s^{2/3}$. To show that having too many disjoint paths across $R_n$, none of which has weight much smaller than typical, we fix a sequence $J=\{j_0,j_1, \ldots, j_{s}\}$ that encodes a path crosses different lines of the grid. We denote by $\gamma_{J}$ the best path with encoding $J$ (after centering). Lemma \ref{l:thin1plus} shows that for appropriate choices of $s$ and $t$, $\ell(\gamma_{J})$ is likely to be rather small, and an application of the BK inequality makes the occurrence of $\ell$ such disjoint paths sufficiently unlikely such that a union bound over all possible $\ell$ tuples of $J$  finishes the proof of Proposition \ref{p:multi}.}
\label{f:grid}
\end{figure}

The next lemma shall show that none of the geodesic lengths from $A'_{n}$ to $B'_n$ can be too small. For $u=(u_1,u_2)\in \Z_{\geq 0}^2$, let us define the function $S(u)=(\sqrt{u_1}+\sqrt{u_2})^{2}$. Recall that, \eqref{e:mean} bounds $|\E T_{u,v}-S(v-u)|$ provided the straight line joining $u$ and $v$ has slope bounded away from $0$ and $\infty$. The next lemma shall show that none of the geodesic lengths from $u\in A'_{n}$ to $v\in B'_n$ is likely to be too small compared to $S(v-u)$.

\begin{lemma}
\label{l:inf}
For each fixed constant $c_1>0$, there exists $c>0$ such that for all $n,\ell$ and $m$ as in the statement of Proposition \ref{c:rare}, we have
$$\P\left(\inf_{u\in A'_{n}, v\in B'_{n}} T_{u,v}-S(v-u) \leq -c_1\ell^{1/4}n^{1/3}\right) \leq  e^{-c\ell^{3/4}}.$$
\end{lemma}

\begin{proof}
Let $\mathscr{I}$ and $\mathscr{J}$ denote the partition of $A'_{n}$ and $B'_{n}$ into intervals of length $2n^{2/3}$. It follows from our hypothesis on $m$ and $\ell$ that for each $I\in \mathscr{I}$ and $J\in \mathscr{J}$, Theorem \ref{t:supinf}, (i) applies to 
$$ \inf_{u\in I, v\in J} T_{u,v}-\E T_{u,v},$$
with a possibly increased value of $\psi$. The result now follows from observing that by \eqref{e:mean}, we have for $u\in I, v\in J$, $|\E T_{u,v}-S(v-u)|\leq Cn^{1/3}$ for some $C>0$ and taking a union bound over all pairs $(I,J)$. 
\end{proof}

Let $\cg_{\ell}$ denote the event that there exists $u_1< u_2 < \cdots < u_{\ell}$ on $A'_{n}$, and $v_1< v_2 < \cdots <v_{\ell}$ on $B'_{n}$, and disjoint paths $\gamma_{i}$ joining $u_i$ and $v_i$ contained in $R_n$ such that $\ell(\gamma_{i}) \geq S(v_i-u_i) -c_1\ell^{1/4} n^{1/3}$. In view of Lemma \ref{l:rec} and Lemma \ref{l:inf} the following proposition suffices to prove Proposition \ref{c:rare}.

\begin{proposition}
\label{p:multi}
In the above set-up, we have $\P(\cg_{\ell})\leq e^{-c\ell^{1/4}}$.
\end{proposition}

We shall need some preparation to prove Proposition \ref{p:multi}. We shall divide the rectangle $R_n$ into an $s\times \ell^{1/8}t$ grid of sub-rectangles, see Figure \ref{f:grid} for an illustration in the special case $m=0$. We shall choose a suitable $s$ and $t$ later. More precisely, consider lines $\mathbf{L}_{i}$ with slope $-1$ equally spaced with internal spacing $\frac{2n}{s}$ such that $\mathbf{L}_{0}=\mathbb{L}_{0}$ and $\mathbf{L}_{s}=\mathbb{L}_{n}$. Observe that each of these lines intersects $R_n$ in a line segment of length $2\ell^{1/8}n^{2/3}$. Abusing notation let $\mathbf{L}_{i}$ denote those line segments, partition the line segment $\mathbf{L}_i$ into $2\ell^{1/8}t$ many equally spaced line segment $L_{i,j}$ each of length $\frac{n^{2/3}}{t}$.

Our next objective is the following. Fix a sequence $J:=\{j_0, j_1,j_2,\ldots , j_{s-1}, j_{s}\}$ taking values in $[-\ell^{1/8}t, \ell^{1/8}t)\cap \Z$. For $u\in I_{j_0}, v\in I_{j_s}$, let $\mathscr{P}(u,v,J)$ denote the set of all paths from $u$ to $v$ that passes through the line segment $L_{i,j_{i}}$ for each $i=0,1,\ldots, s$. Let $\gamma_{J}$ denote the path that maximizes
$$\ell(\gamma)-S(v-u)$$
over all $\gamma\in \mathscr{P}(u,v,J)$ and over all $u\in I_{j_0},v\in I_{j_s}$; let $u_{J}$ and $v_{J}$ denote the starting and ending point of $\gamma_{J}$. We shall show that for suitable choices of parameters $\ell(\gamma_{J})$ is typically much smaller than $S(v_{J}-u_{J})$; more specifically we have the following lemma.

%

\begin{lemma}
\label{l:thin1plus}
For any $\psi<1$, there exists $c_0>0$ sufficiently small and $c_1>0$ such that for all $n,\ell,m$ as in the set-up of Proposition \ref{c:rare}, for all $s\leq \sqrt{\ell}$ sufficiently large and $c_0t=s^{2/3}$, we have for each $J$ as above and $\gamma_{J}$ as above with  endpoints $u_{J}$ and $v_{J}$
$$\P(\ell(\gamma_{J})-S(v_{J}-u_{J})\geq -c_1 s^{2/3}n^{1/3})\leq e^{-cs^{1/2}}$$
for some $c=c(\psi)>0$.
\end{lemma}

Let us postpone the proof of Lemma \ref{l:thin1plus} for now and use it first to complete the proof of Proposition \ref{p:multi}. Let $c_0$ be fixed such that the conclusion of Lemma \ref{l:thin1plus} holds, and let us set $s=\sqrt{\ell}$ and $t=\frac{s^{2/3}}{c_0}$. We need to control the entropy of the  $\ell$-tuples  $(J_1, J_2, \ldots, J_{\ell})$ of sequences associated with $\ell$ disjoint paths as is predicated to exist on the event $\cg_{\ell}$. Let $\cJ$ denote the set of all sequences $J$ as described above. For any path $\gamma$ from $A'_{n}$ to $B'_{n}$, let $J(\gamma)=(j_{0}, \ldots, j_{s})$ denote the element in $\cJ$ such that $\gamma$ passes through the line segment $L_{i,j_{i}}$ for each $i=0,1,\ldots, s$. To this end we have the following lemma.

\begin{lemma}
\label{l:count}
There exists a deterministic set $\mathcal{C}=\mathcal{C}_{\ell,t,s} \subseteq \cJ ^{\ell}$ with
$$|\mathcal{C}|\leq (\ell+2\ell^{1/8}t)^{2\ell^{1/8}t(s+1)},$$
such that on the event $\cg_{\ell}$, there exists $(J_1,J_2,\ldots, J_{\ell})\in \mathcal{C}$ such that
$J(\gamma_i)=J_{i}$ for each $i=1,2,\ldots ,\ell$.
\end{lemma}

\begin{proof}
On $\cg_{\ell}$, let $\gamma_1,\gamma_2, \ldots, \gamma_{\ell}$ be a naturally ordered set of disjoint paths as given by the definition of the event.
For each  $i\in \{1,2,\ldots , \ell\}$, let $J_{i}=(j^{(i)}_{0}, \ldots, j^{(i)}_{s})=J(\gamma_i)$ be the element of $\cJ$ such that $\gamma_{i}$ intersects $L_{k,j^{(i)}_{k}}$ for each $k$. We need to bound the total number of all possible such tuples $(J_1,J_2,\ldots, J_{\ell})$. Observe that the ordering implies, if $i_1<i_2$, we must have $j^{(i_1)}_{k}\leq j^{(i_2)}_{k}$ for each $k$.
 It follows that $\mathcal{C}$ can be enumerated by picking $(s+1)$ many the non-decreasing sequences of length $\ell$ where each co-ordinate takes values in $-\ell^{1/8}t$ to $\ell^{1/8}t$. So our task is reduced to enumerating integer sequences $-\ell^{1/8}t \leq y_1 \leq y_2 \leq \cdots \leq  y_{\ell} \leq \ell^{1/8}t$. By looking at the difference sequence $z_{k}=(y_{k}-y_{k-1})$ (and setting $y_0=-\ell^{1/8}t$) this reduces to enumerating {non-negative} integer sequence $\{z_i\}$ with $z_1+z_2+\cdots +z_{\ell} \leq 2\ell^{1/8}t$. It is a standard counting exercise ({see, e.g.\ the discussion following \cite[(1.19)]{Sta11}}) to see that number of such sequences is bounded by $\binom {\ell+2\ell^{1/8}t}{2\ell^{1/8}t}$. By choosing $(s+1)$ many such sequences, we get an upper bound of $\binom {\ell+2\ell^{1/8}t}{2\ell^{1/8}t}^{s+1}$ and the result follows.
\end{proof}

We can now complete the proof of Proposition \ref{p:multi}.

\begin{proof}[Proof of Proposition \ref{p:multi}]
For a fixed sufficiently large $\ell<n^{0.01}$, set $s=\ell^{1/2}$ and let $t=\frac{s^{2/3}}{c_0}$ as in the statement of Lemma \ref{l:thin1plus}. For $\mathcal{C}$ as in Lemma \ref{l:count} and $(J_1,J_2,\ldots , J_{\ell})\in \mathcal{C}$ let $A_{J_1,J_2,\ldots , J_{\ell}}$ denote the event that there exist disjoint paths $\gamma_1, \gamma_2, \ldots, \gamma_{\ell}$ satisfying the condition in the definition of $\cg_{\ell}$ with $J_i=J(\gamma_i)$ (in particular, by definition of $\gamma_{J_{i}}$ and $\cg_{\ell}$ this implies that $\ell(\gamma_{J_{i}})-S(v_{J_i}-u_{J_i}) \geq -c_1\ell^{1/4}n^{1/3}$ where $u_{J_{i}}$ and $v_{J_{i}}$ are starting an ending point of $\gamma_{J_{i}}$ respectively). Using Lemma \ref{l:count}, it follows that $\P(\cg_{\ell})$ is upper bounded by
$$\sum_{(J_1,J_2,\ldots , J_{\ell})\in \mathcal{C}} \P(A_{J_1,J_2,\ldots , J_{\ell}}).$$

Now observe that for any path $\gamma$ from $u$ to $v$, the event that $\ell(\gamma)-S(v-u) \geq -c_1\ell^{1/4}n^{1/3}$ is increasing in the vertex weights and hence by the BK inequality ({see \cite{AGH18} for the variant used here}) the probability of a number of such events happening disjointly is upper bounded by the product of the marginal probabilities. Hence we have 
$$\P(A_{J_1,J_2,\ldots , J_{\ell}}) \leq \prod_{i=1}^{\ell} \P\left(\ell(\gamma_{J_{i}})-S(v_{J_i}-u_{J_i})\geq -c_1\ell^{1/4}n^{1/3}\right) \leq e^{-c\ell^{5/4}}$$
where the final inequality follows from Lemma \ref{l:thin1plus}. By Lemma \ref{l:count} (and our choices of $t$ and $s$) it follows that for any $\epsilon>0$ we have $|\mathcal{C}|\leq \ell^{\ell^{23/24+\epsilon}}$ and hence the result follows, by summing over all elements of $\mathcal{C}$.
\end{proof}

\subsection{Proof of Lemma \ref{l:thin1plus}}
We shall now provide the proof of Lemma \ref{l:thin1plus}. Let and $s,t$ and $J$ be fixed as in the statement of the lemma. For any $\gamma\in \mathscr{P}(u,v,J)$, setting $u_0=u$ and $u_{s}=v$, and $u_{i}$ to be the point where $\gamma$ intersects $L_{i,j_{i}}$ we have
$$ \ell(\gamma)- S(v-u) = \sum_{i=0}^{s-1}\left(T_{u_i,u_{i+1}}-S(u_{i+1}-u_{i})\right)+ \sum_{i=0}^{s-1}S(u_{i+1}-u_{i}) -S(u_{s}-u_0).$$
It therefore follows that for $\gamma_{J}$ as in the statement of the lemma with endpoints $u_{J}$ and $v_{J}$ we have
$$\ell(\gamma_J)- S(v_{J}-u_{J})\leq \sum_{i=0}^{s-1}\left( \sup_{u_{i}\in L_{i,j_i}, v_{i}\in L_{i+1,j_{i+1}}} T_{u_i,v_i}-S(v_i-u_i)\right)+\mathcal{L}(J)$$
where 
$$\mathcal{L}(J):= \sup_{u_{i}\in L_{i,j}} \sum_{i=0}^{s-1} S(u_{i+1}-u_i) - S(u_s-u_0).$$
We shall prove Lemma \ref{l:thin1plus} by controlling each of the terms above separately; $\mathcal{L}(J)$ is easy to control.

\begin{lemma}
\label{l:concave}
In the above set-up, for all $J$ as above we have 
$$\mathcal{L}(J)\leq 0.$$ 
\end{lemma}

\begin{proof}
This follows immediately by observing that $S(u)=(\sqrt{u_1}+\sqrt{u_2})^2$ is a concave function on $\R^2_{+}$. 
\end{proof}

The next task is to control the terms $\sup_{u_{i}\in L_{i,j_i}, v_{i}\in L_{i+1,j_{i+1}}} T_{u_i,v_i}-S(v_i-u_i)$ for each $i$. Towards this we have the following lemma.





\begin{lemma}
\label{l:mean}
In the above set-up, with $s\leq \sqrt{\ell}$ sufficiently large, for $c_0$ sufficiently small (depending on $\psi$) and $t=\frac{s^{2/3}}{c_0}$, we have for each $J$ and for each $i\in \{0,1,\ldots, s-1\}$, 
$$\E \left[\sup_{u_{i}\in L_{i,j_i}, v_{i}\in L_{i+1,j_{i+1}}} T_{u_i,v_i}-S(v_i-u_i) \right] \leq -C' (n/s)^{1/3}$$
for some $C'>0$.
\end{lemma}

Postponing the proof of Lemma \ref{l:mean} for now, let us first complete the proof of Lemma \ref{l:thin1plus}. Fix $s\leq \sqrt{\ell}$ sufficiently large, $c_0$ sufficiently small and $t=\frac{s^{2/3}}{c_0}$ such that the conclusion of Lemma \ref{l:mean} holds. Let us also fix sequence $J:=\{j_0, j_1,j_2,\ldots , j_{s-1}, j_{s}\}$ taking values in $[-\ell^{1/8}t, \ell^{1/8}t)\cap \Z$, and recall that $\gamma_{J}$ denotes the path from $A'_n$ to $B'_{n}$ that passes through the line segment $L_{i,j_{i}}$ for each $i=0,1,\ldots, s$ and maximizes 
$\sup_{u\in A'_n,v\in B'_n} \ell(\gamma)-S(v-u)$ among all such paths with end points $u$ and $v$ and also over all possible pairs $u,v$.

\begin{proof}[Proof of Lemma \ref{l:thin1plus}]
For notational convenience, let us set $$Z_i:=\sup_{u_{i}\in L_{i,j_i}, v_{i}\in L_{i+1,j_{i+1}}} T_{u_i,v_i}-S(v_i-u_i)$$
and let $Z_i'=(n/s)^{-1/3} Z_i$. Observe that, by our choice of $s$, Theorem \ref{t:supinf}, (ii) applies to $Z_{i}$ for all $i$ and hence it follows that $\{Z'_{i}\}$ is a sequence of independent subexponential random variables. {Applying the Bernstein inequality for sums of independent subexponential random variables (see e.g.\ \cite[Corollary 2.8.3]{Ver18})} to $\sum Z'_{i}$ it follows that for $\kappa>0$ and $s$ sufficiently large we have
$$\P\left(\sum_i Z'_i-\E Z'_{i} \geq \kappa s\right)  \leq e^{-cs}$$
for some $c>0$. By Lemma \ref{l:mean} it follows that 
$$\sum_{i} \E Z'_{i}\leq -C's$$ for $C'>0$ as in Lemma \ref{l:mean}. By setting $\kappa=\frac{C'}{2}$ we get that 
$$\P\left(\sum_{i}Z_{i}\geq -\frac{C'}{2}s^{2/3}n^{1/3}\right) \leq e^{-cs}$$
for some $c>0$. The proof of the lemma is complete, noticing
$$\ell(\gamma_{J})-S(v_{J}-u_{J})\leq \sum_{i} Z_{i} +\mathcal{L}(J),$$
using Lemma \ref{l:concave} and and finally setting $c_1=C'/10$. 
\end{proof}



It remains to prove Lemma \ref{l:mean}. We prove the following general result which immediately implies Lemma \ref{l:mean}. 

\begin{lemma}
\label{l:meangeneral}
Let $A_*$ (resp. $B_{*}$) denote the line segment parallel to the line $x+y=0$ of length $c_0n^{2/3}$ with centre $(-mn^{2/3},mn^{2/3})$ (resp.\ $\mathbf{n}$). For $\psi<1$ , there exists $c_0$ sufficiently small (depending on $\psi$) such that for $|m|<\psi n^{1/3}$ and for all $n$ sufficiently large we have
$$\E \sup_{u\in A_*, v\in B_*} T_{u,v}-S(v-u) \leq- C'n^{1/3}$$
for some $C'>0$.
\end{lemma}

\begin{figure}[htbp!]
\centering
\includegraphics[width=0.25\textwidth]{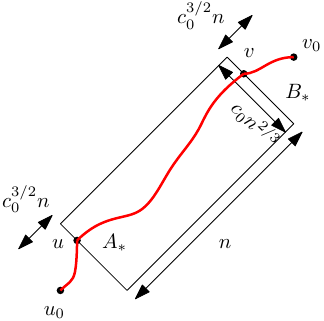}
\caption{An illustration of the proof of Lemma \ref{l:meangeneral}. We define points $u_0$ and $v_0$ at distance $c_0^{3/2}n$ away from $A_{*}$ and $B_{*}$ respectively in the diagonal direction. The negativity of the mean of Tracy-Widom distribution implies $\E T_{u_0,v_0}-S(v_0-u_0)\leq -cn^{1/3}$ for some $c>0$. Using $T_{u_0,v_0}\geq T_{u_0,u}+T_{u,v}+T_{v,v_0}$ for any $u\in A_{*}, v\in B_{*}$, and choosing $c_0$ small this leads to $\E \sup_{u\in A_*, v\in B_*} T_{u,v}-S(v-u) \leq- C'n^{1/3}$ for some $C'>0$ completing the proof of the lemma.} 
\label{f:thinmean}
\end{figure}

\begin{proof}
Consider the straightline joining the points $(-mn^{2/3},mn^{2/3})$ and $\mathbf{n}$ and let $u_0$ (resp.\ $v_0$) denote the point where it intersects the line $x+y=-2c_0^{3/2}n$ (resp.\ $x+y=2n+2c_0^{3/2}n$); see Figure \ref{f:thinmean} for an illustration. Clearly, we have for any $u\in A_*,v\in B_*$
$$T_{u,v}-S(v-u)\leq T_{u_0,v_0}-S(v_0-u_0) - (T_{u_0,u}-S(u-u_0))-(T_{v,v_0}-S(v_0-v))+\widetilde{S}(u_0,u,v,v_0)$$
where 
$$\widetilde{S}(u_0,u,v,v_0):=S(v_0-u_0)-S(u-u_0)-S(v-u)-S(v_0-v).$$
For notational convenience, let us define:
$$\mathbb{A}:=\E \inf_{u\in A_*} (T_{u_0,u}-S(u-u_0));$$
$$\mathbb{B}:=\E \inf_{v\in B_*} (T_{v,v_0}-S(v_0-v)).$$
Clearly
$$\E \sup_{u\in A_*, v\in B_*} T_{u,v}-S(v-u)\leq \E T_{u_0,v_0}-S(v_0-u_0)-\mathbb{A}-\mathbb{B}+\sup_{u,v} \widetilde{S}(u_0,u,v,v_0).$$
Using Theorem \ref{t:tw} (and the fact that the weak convergence there is uniform in $h\in K$ for every compact $K\subset (0,\infty)$), Theorem \ref{t:onepoint} and the well-known fact that GUE Tracy-Widom distribution has negative mean (see, e.g.\ \cite[Lemma A.4]{BGHH20} for a proof), it follows that 
\begin{equation}
\label{e:meanneg}
\E T_{u_0,v_0}-S(v_0-u_0)\leq -cn^{1/3}
\end{equation}
for some constant $c>0$ (depending only on $\psi$ and in particular not depending on $c_0$) for $n$ sufficiently large. We would now be done (by choosing $c_0$ small) if we could show that each of the other three terms $-\mathbb{A}, -\mathbb{B}$ and $\sup_{u,v}\widetilde{S}(u_0,u,v,v_0)$ can be upper bounded by $Cc_0^{1/2}n^{1/3}$ for some absolute constant $C$. This is what we shall do.

For $\mathbb{A}$ and $\mathbb{B}$, notice that Theorem \ref{t:onepoint} (more precisely \eqref{e:mean}) implies that $\inf_{u\in A_*} \E T_{u_0,u}-S(u-u_0), \inf_{v\in B_*} \E T_{v,v_0}-S(v_0-v) \geq  -C(c_0^{3/2}n)^{1/3}$ for some constant $C>0$. Together with Theorem \ref{t:supinf}, it implies that $\mathbb{A}, \mathbb{B}\geq -Cc_0^{1/2}n^{1/3}$ for some (possibly different) $C>0$. Finally, it is simple algebra, using the definition of $S$, to verify that $\sup_{u,v}\widetilde{S}(u_0,u,v,v_0)\leq Cc_0^{1/2}n^{1/3}$ for some $C>0$. Putting these together and choosing $c_0$ sufficiently small completes the proof of the lemma.
\end{proof}

\subsection{Most geodesics coalesce quickly}
\label{s:quick}
As mentioned earlier, for our purposes we need a stronger variant of Theorem \ref{t:rare} and Proposition \ref{c:rare} showing that in addition to disjoint geodesics being unlikely, most pairs of geodesics between points on $A_n$ and $B_n$ (in the notation of Theorem \ref{t:rare}) actually merge together rather quickly. To state the result formally, we introduce the following terminology. Let $A^*_{n}$ (resp.\ $B^*_{n}$) denote the line segment along the line $x+y=0$, denoted $\mathbb{L}_0$ (resp.\ $\mathbb{L}_{n}$) of length $2Hn^{1/3}$ with midpoint $(-mn^{2/3},mn^{2/3})$ (resp.\ $\mathbf{n}$). For $u,u'\in A^*_{n}$ and $v,v'\in B^*_{n}$ we say that $(u,v)\sim (u',v')$ if the geodesics $\Gamma_{u,v}$ and $\Gamma_{u',v'}$ coincide between the lines $\mathbb{L}_{n/3}$ and $\mathbb{L}_{2n/3}$. It is easy to see that $\sim$ is an equivalence relation. Let $M_{n}$ denote the number of equivalence classes. We have the following theorem.


\begin{theorem}
\label{t:multi2}
For $\psi<1$ and $H>0$, there exists $c>0$ such that for all $m$ with $|m|<\psi n^{1/3}$, all $\ell< n^{0.01}$ sufficiently large and all $n\in \N$ sufficiently large we have
$$\P(M_n\geq \ell)\leq e^{-c\ell^{1/128}}.$$
\end{theorem}

Observe that the only significant difference between Proposition \ref{c:rare} and Theorem \ref{t:multi2} is that in the former case we were considering disjoint paths from $A'_n$ to $B'_n$ which were naturally ordered. In the set-up of Theorem \ref{t:multi2}, we have to consider also geodesics that can potentially cross each other. To circumvent this issue we show that if there exists a large number of different equivalence classes there must be some stretch of linear length between $\mathbb{L}_0$ and $\mathbb{L}_{n}$ such that the geodesics corresponding to the equivalence classes are disjoint in this stretch.  For the purpose of this proof we shall always assume that we are working on the probability one subset on which there is a unique geodesic between every pair of vertices in $\Z^2$ without explicitly mentioning the same.

We first need a combinatorial lemma. For $a<b<c<d\in \Z$, let us consider the parallel lines $\mathbb{L}_{a}$, $\mathbb{L}_b$, $\mathbb{L}_c$ and $\mathbb{L}_d$. Let $u_1<u_2<\cdots <u_{k}$ (resp.\ $v_1<v_2<\cdots <v_{k}$) be points on $\mathbb{L}_a$ (resp.\ $\mathbb{L}_d$) such that for $i\neq i'$, $\Gamma_{u_i,v_i}$ and $\Gamma_{u_{i'},v_{i'}}$ do not coincide between $\mathbb{L}_b$ and $\mathbb{L}_c$. We have the following result. 

\begin{lemma}
\label{l:ordering1}
In the above set-up, there exists a subset $I$ of $\{1,2,\ldots, k\}$ with $|I|\geq (k-1)/3$ such that the restrictions of $\{\Gamma_{u_i,v_i}\}_{i\in I}$ are disjoint between at least one of the following three pairs of lines:  (i) $\mathbb{L}_{a}$ and $\mathbb{L}_{b}$,  (ii) $\mathbb{L}_{b}$ and $\mathbb{L}_{c}$, (iii) $\mathbb{L}_{c}$ and $\mathbb{L}_{d}$.
\end{lemma}

\begin{proof}
Notice that the planar ordering of the geodesics together with their assumed uniqueness implies that the geodesics $\Gamma_{u_i,v_i}$ cannot cross one another. For $i=1,2,\ldots k-1$, let us define $H_{i}=(h_i^1,h_i^2,h_i^{3})\in \{0,1\}^3$ as follows: $h_{i}^{1}=1$ if $\Gamma_{u_{i},v_i}$ and $\Gamma_{u_{i+1},v_{i+1}}$ has a common vertex between $\mathbb{L}_{a}$ and $\mathbb{L}_b$ and 0 otherwise; $h_{i}^{2}=1$ if $\Gamma_{u_{i},v_i}$ and $\Gamma_{u_{i+1},v_{i+1}}$ has a common vertex between $\mathbb{L}_{b}$ and $\mathbb{L}_c$ and 0 otherwise; $h_{i}^{3}=1$ if $\Gamma_{u_{i},v_i}$ and $\Gamma_{u_{i+1},v_{i+1}}$ has a common vertex between $\mathbb{L}_{c}$ and $\mathbb{L}_d$ and 0 otherwise. Observe that if for some $i$, $H_{i}=(1,1,1)$ this implies $\Gamma_{u_i,v_i}$ and $\Gamma_{u_{i+1},v_{i+1}}$ coincide between $\mathbb{L}_b$ and $\mathbb{L}_{c}$ (using the ordering described above) which contradicts the hypothesis. Therefore, there exists at least one $0$ in each $H_{i}$ and consequently there exists $I\subset \{1,2,\ldots, k\}$ with $|I|\geq (k-1)/3$ and $j\in \{1,2,3\}$ such that $h_{i}^{j}=0$ for all $i\in I$. We claim that $I$ satisfies the conclusion of the lemma, that is $\{\Gamma_{u_i,v_i}\}_{i\in I}$ are disjoint between at least one of the following three pairs of lines:  (i) $\mathbb{L}_{a}$ and $\mathbb{L}_{b}$,  (ii) $\mathbb{L}_{b}$ and $\mathbb{L}_{c}$, (iii) $\mathbb{L}_{c}$ and $\mathbb{L}_{d}$. Indeed, we shall show that if $j=1$, $\{\Gamma_{u_i,v_i}\}_{i\in I}$ are disjoint between $\mathbb{L}_{a}$ and $\mathbb{L}_{b}$. The cases $j=2$ and $j=3$ can be handles in an identical manner. For $j=1$, observe that if $i,i'\in I$ is such that $i'=i+1$, $\Gamma_{u_i,v_{i}}$ and   $\Gamma_{u_{i'},v_{i'}}$ are disjoint between $\mathbb{L}_a$ and $\mathbb{L}_b$ by definition. If $i'>i+1$, then again, by planar ordering neither $\Gamma_{u_i,v_{i}}$ nor  $\Gamma_{u_{i'},v_{i'}}$ can cross $\Gamma_{u_{i+1},v_{i+1}}$ and hence their restrictions between $\mathbb{L}_a$ and $\mathbb{L}_b$ are disjoint by definition of $I$. This completes the proof. 
%
%
%
\end{proof}

To handle paths that are not ordered in the set-up of Theorem \ref{t:multi2}, we invoke the Erd\H{o}s-Szekeres Theorem and use Lemma \ref{l:ordering1} to obtain the following result.

\begin{lemma}
\label{l:ordering}
In the set-up of Theorem \ref{t:multi2}, suppose $u_1<u_2<\cdots < u_{k}$ (resp.\ $v_1,v_2,\ldots , v_{k}$ with $v_i\neq v_j$) be points on $\mathbb{L}_0$ (resp.\ $\mathbb{L}_n$) such that each of the pairs $(u_i,v_{i})$ are from a different equivalence class, i.e., for each pair $(i,j)$ with $i\neq j$, $\Gamma_{u_{i}, v_{i}}$ and $\Gamma_{u_{j},v_{j}}$ do not coincide between the lines $\mathbb{L}_{n/3}$ and $\mathbb{L}_{2n/3}$. For $k$ sufficiently large, there exists a subset $I\subset \{1,2,\ldots, k\}$ with $|I|\geq \frac{k^{1/8}}{100}$ that the restrictions of $\{\Gamma_{u_i,v_i}\}_{i\in I}$ are disjoint between at least one of the following four pairs of lines:  (i) $\mathbb{L}_{0}$ and $\mathbb{L}_{n/6}$,  (ii) $\mathbb{L}_{n/6}$ and $\mathbb{L}_{n/3}$, (iii) $\mathbb{L}_{n/3}$ and $\mathbb{L}_{2n/3}$,  (iv) $\mathbb{L}_{2n/3}$ and $\mathbb{L}_{n}$.
\end{lemma}

\begin{proof}
By the Erd\H{o}s-Szekeres theorem ({which states that among $n^{2}+1$ distinct real numbers there always exists a monotone subsequence of length $n+1$, see e.g.\ \cite{Ste95}}), there exists a subset $I_1=\{i_1, i_2, \ldots i_{\sqrt{k}}\}$ of $\{1,2,\ldots, k\}$ such that either $v_{i_1}<v_{i_2}<\cdots <v_{i_{\sqrt{k}}}$ or $v_{i_1}>v_{i_2}>\cdots >v_{i_{\sqrt{k}}}$ (assume momentarily that $k^{1/8}$ is a positive integer). In the former case, applying Lemma \ref{l:ordering1}, we get that there exists $I\subset I_1$ with $|I|\geq (\sqrt{k}-1)/3$ satisfying the conclusion of lemma. completing the proof in that case.


Suppose now the contrary, i.e., $v_{i_1}>v_{i_2}>\cdots >v_{i_{\sqrt{k}}}$. Consider then the points $w_{i_1},w_{i_2},\ldots ,w_{i_{\sqrt{k}}}$ such that  
$\Gamma_{u_{i_{j}}, v_{i_{j}}}$ intersects $\mathbb{L}_{n/6}$ at $w_{i_j}$. 
Notice that these points need not be distinct. However, observe that there exists a further subset $I'$ of $I_1$ with $|I'|\geq k^{1/4}$ such that either: (i) $w_{i}=w_{i'}$ for all $i,i'\in I'$; or, (ii), $w_{i}$ are all distinct for $i\in I'$. For (i), notice that by  our hypothesis, for $i\neq i'\in I_2$, $\Gamma_{w_i,v_i}$ and $\Gamma_{w_{i'},v_{i'}}$ do not coincide between $\mathbb{L}_{n/3}$ and $\mathbb{L}_{2n/3}$, hence (again using the planar ordering and uniqueness of the geodesics) they must be disjoint between $\mathbb{L}_{2n/3}$ and $\mathbb{L}_{n}$ and hence we are done. For (ii), use the Erd\H{o}s-Szekeres theorem to obtain a further subset $I''=\{i'_1<i'_2<\cdots <i'_{k^{1/8}}\}$ of $I'$ such that either 
(ii)(a) $w_{i'_1}< w_{i'_2}< \cdots < w_{i'_{k^{1/4}}}$ or (ii)(b)  $w_{i'_1}> w_{i'_2}> \cdots > w_{i'_{k^{1/4}}}$. Notice that, in scenario (ii)(a), by planar ordering, each pairs geodesics $\Gamma_{w_{i},v_{i}}$  and $\Gamma_{w_{i'},v_{i'}}$ must intersect for $i,i'\in I''$ and hence by uniqueness of geodesics, $\Gamma_{u_{i},w_{i}}$ are pairwise disjoint, completing the proof. For (ii)(b) notice that the geodesics $\Gamma_{w_i,v_{i}}$ are ordered for $i\in I''$ and applying Lemma \ref{l:ordering1}, we get a further subset $I$ of size $(k^{1/8}-1)/3$ having the desired property. Note that, in the general case where $k$ is not assumed to be a perfect eighth power, the same argument gives an $I$ with $ |I|\geq (\lfloor\lfloor k^{1/2}\rfloor^{1/4} \rfloor -1)/3\geq \frac{k^{1/8}}{100}$ for $k$ sufficiently large. This completes the proof of the lemma. 
%
%
%
%
%
%
\end{proof}

We can now prove Theorem \ref{t:multi2}.

\begin{proof}[Proof of Theorem \ref{t:multi2}]
Fix $n$ sufficiently large and $\ell< n^{0.01}$ sufficiently large. Since the probability upper bound we are seeking is $\gg e^{-\ell^{1/4}}$ for $\ell$ large,  using Lemma \ref{l:rec} we can restrict ourselves on the event that none of the geodesics from $A^*_n$ to $B^*_n$ exit $R_n$. On the event $\{M_n\geq \ell\}$, let $\{\Gamma_{u_{i},v_{i}}\}_{i=1}^{\ell}$ denote $\ell$ geodesics such that each $(u_i,v_i)$ is from a different equivalence class. It is easy to observe, that one of the following must hold: (a) there exists a subset $I$ of $\{1,2,\ldots, \ell\}$ with $|I|\geq \ell^{1/4}$ such that as $i\in I$ is varied, $u_i$s are all distinct and $v_{i}$s are all the same; (b) there exists a subset $I$ of $\{1,2,\ldots, \ell\}$ with $|I|\geq \ell^{1/4}$ such that as $i\in I$ is varied, $u_i$s are all the same and $v_{i}$s are all distinct; (c) there exists a subset $I$ of $\{1,2,\ldots, \ell\}$ with $|I|\geq \ell^{1/4}$ such that as $i\in I$ is varied, $u_i$s are all distinct and $v_{i}$s are all distinct. In scenario (a), the planar ordering (plus the hypothesis that $(u_i,v_i)$ are all in different equivalence classes) forces that the restrictions of $\Gamma_{u_{i},v_{i}}$ and $\Gamma_{u_{i'},v_{i'}}$ between $\mathbb{L}_0$ and $\mathbb{L}_{n/3}$ are disjoint for any distinct $i,i'\in I$. Since these restrictions are also contained within $R_{n}$, Proposition \ref{c:rare} applies and upper bound the probability of this scenario by $e^{-c\ell^{1/16}}$. In scenario (b), we conclude by an identical argument that the restrictions of $\Gamma_{u_{i},v_{i}}$ and $\Gamma_{u_{i'},v_{i'}}$ between $\mathbb{L}_{2n/3}$ and $\mathbb{L}_{n}$ are disjoint for any distinct $i,i'\in I$, and upper bound the probability of this scenario by $e^{-c\ell^{1/16}}$ using Proposition \ref{c:rare} again. In scenario (c), we apply Lemma \ref{l:ordering} to find a subset $I'$ of $I$ of size at least $\frac{\ell^{1/32}}{100}$ such that the restrictions of $\{\Gamma_{u_{i},v_{i}}\}_{i\in I'}$ are pairwise disjoint either between $\mathbb{L}_0$ and $\mathbb{L}_{n/6}$ or between $\mathbb{L}_{n/6}$ and $\mathbb{L}_{n/3}$ or between $\mathbb{L}_{n/3}$ and $\mathbb{L}_{2n/3}$ or between $\mathbb{L}_{2n/3}$ and $\mathbb{L}_{n}$. As $\ell$ is sufficiently large (compared to $H$), Proposition \ref{c:rare} applies in each of these cases and we conclude that the probability of scenario (c) is upper bounded by $e^{-c\ell^{1/128}}$ for some $c>0$. The proof of the theorem is completed by taking a union bound over scenarios (a), (b) and (c).
%
\end{proof}

%

We finish this section with a remark on the midpoint problem which was alluded to before. Consider the geodesic $\Gamma_{n}$ from $\mathbf{0}$ to $\mathbf{n}$. Assuming $n\in 2\N$ what is the probability that $\Gamma_{n}$ passes through the midpoint $\mathbf{\frac{n}{2}}$? This question was asked in \cite{BKS04} in the context of first passage percolation and became popular as the ``midpoint problem". It is natural to conjecture that the probability is $\Theta(n^{-2/3})$ for models in KPZ universality class where the transversal fluctuation exponent is believed to be $2/3$. However, for non-integrable models, to show even that this probability is $o(1)$ remained open for many years and was only recently settled in \cite{AH16}. The optimal upper bound described above can easily be deduced from Theorem \ref{t:multi2} as the following remark sketches.

\begin{remark}
\label{r:midpoint} 
Theorem \ref{t:multi2} implies that the number of vertices on the line $x+y=n$ contained in a geodesic from $(t,-t)$ to $(n+t,n-t)$ for some $t$ with $|t|\leq n^{2/3}$ has expectation uniformly bounded above by a constant. This, together with the translation invariance of the model gives that 
$n^{2/3}\times \P(\mathbf{\frac{n}{2}}\in \Gamma_{n})=O(1)$ and hence $\P(\mathbf{\frac{n}{2}}\in \Gamma_{n})=O(n^{-2/3})$. For the lower bound, it suffices to conclude that with probability bounded away from 0 the following event occurs: the geodesic from $(n^{2/3},-n^{2/3})$ to $(n+n^{2/3},n-n^{2/3})$ and the geodesic from $(-n^{2/3},n^{2/3})$ to $(n-n^{2/3},n+n^{2/3})$ intersect the line $x+y=n$ at the same point $(n/2+t,n/2-t)$ for some $t$ with $|t|\leq n^{2/3}$. See \cite{BB20} for the details and further extensions along this approach.
%
\end{remark}

\section{A geodesic not directed axially hitting the origin is unlikely}
\label{s:h}

We shall prove Proposition \ref{p:o1} in this section. Fix $h\in (0,1)$ sufficiently small. Recall that Proposition \ref{p:o1} seeks to bound the probability of the event $\ce_{n,h}$ that there exists a geodesic from $u$ located on the left or bottom side of the square $[-n,n]^2$ ($\mathsf{Ent}_{n}$) to $w$ located on the top or right of the square $[-n,n]^2$ ($\mathsf{Exit}_{n}$) with $\mbox{slope}(u,w)\in (\frac{h}{2}, \frac{2}{h})$ passing through $\mathbf{0}$. Instead first we show that for any fixed $n$ sufficiently large there exists a (deterministic) point $v\in [-\frac{nh}{100}, \frac{nh}{100}]^2$ such that the existence a geodesic as above passing through $v$ instead of $\mathbf{0}$ has probability $O(n^{-1/3})$. 

We discretize $\mathsf{Ent}_{n}$ and $\mathsf{Exit}_{n}$ into sub-intervals of length $n^{2/3}$ each. Notice first that for $h$ sufficiently small any path from $\mathsf{Ent}_{n}$ to $\mathsf{Exit}_{n}$ that intersects $[-\frac{nh}{100}, \frac{nh}{100}]^2$ must start rather close to the third quadrant of $\Z^2$ and end rather close to the first quadrant of $\Z^2$. More specifically for
$$i=-hn^{1/3}, -hn^{1/3}+1, \ldots, 0, 1,2,\ldots , n^{1/3},$$ let $\cB_{i}$ denote the line segment $[-in^{2/3},(-i+1)n^{2/3}]\times \{-n\}$ and let $\cL_{i}$ denote the line segment $ \{-n\}\times [-in^{2/3},(-i+1)n^{2/3}]$. Similarly,  let $\cT_{i}$ denote the line segment $[(i-1)n^{2/3},in^{2/3}]\times \{n\}$ and let $\cR_{i}$ denote the line segment $ \{n\}\times [(i-1)n^{2/3},in^{2/3}]$. For $I\in \cup_{i} \cB_{i}\bigcup \cup_{i} \cL_{i}$ and $J\in \cup_{i} \cT_{i} \bigcup \cup_{i} \cR_{i}$, we say that the pair $(I,J)$ is $h$-compatible if the straight line joining the mid-point of $I$ to the midpoint of $J$ has slope in $(\frac{h}{10}, \frac{10}{h})$; see Figure \ref{f:highway}. It is easy to see that, for $n$ sufficiently large, if there exists $u\in \mathsf{Ent}_{n}$, $w\in \mathsf{Exit}_{n}$ such that $\mbox{slope}(u,w)\in (\frac{h}{4}, \frac{4}{h})$ and there exists a directed path from $u$ to $w$ that intersects $[-\frac{nh}{100}, \frac{nh}{100}]^2$ then there exists $I\in \cup \cB_{i} \cup \cL_{i}$, $J\in \cup \cT_{i} \cup \cR_{i}$ with $u\in I,w\in J$ and $(I,J)$ is $h$-compatible. The following lemma is the key to the proof of Proposition \ref{p:o1}.


\begin{figure}[htbp!]
\centering
\includegraphics[width=0.4\textwidth]{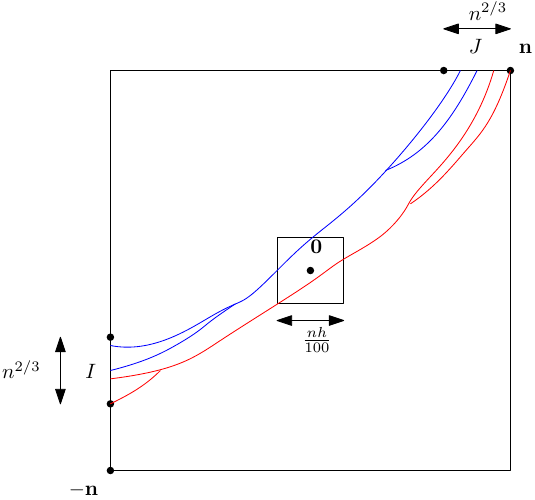}
\caption{The bottom and left side ($\mathsf{Ent}_{n}$) and the top and right side ($\mathsf{Exit}_{n}$) of the square $[-n,n]^2$ is divided into line segments of length $n^{2/3}$. Fix a pair of such line segments $(I,J)$; $I$ from the bottom and left side, and $J$ from the top and right side. Lemma \ref{l:sidetoside} shows (following arguments as in the proof of Theorem \ref{t:multi2}) that geodesics from $I$ to $J$ coalesce into $O(1)$ many highways before passing near the origin and hence the expected number of vertices on such geodesics in a small (but linear size) box around $\mathbf{0}$ is $O(n)$. Taking a union bound over $O(n^{2/3})$ many pairs of intervals $(I,J)$, we concludes that the number of vertices in a small linear size box around the origin that also lie on geodesics across the square is $O(n^{5/3})$, and the proof of Proposition \ref{p:o1} is concluded using an averaging argument.}
\label{f:highway}
\end{figure}



\begin{lemma}
\label{l:sidetoside}
Let $h\in (0,1)$ be fixed and sufficiently small. There exists a constant $C=C(h)>0$ such that for each $h$-compatible pair of line segments $(I,J)$ we have the following. Let $N=N_{n}(I,J)$ denote the number of vertices $v$ in $[-\frac{nh}{100}, \frac{nh}{100}]^2$ such that there exists $u\in I$ and $w\in J$ with $v\in \Gamma_{u,w}$. Then $\E N_{n}(I,J) \leq Cn$.
\end{lemma}

We postpone the proof of Lemma \ref{l:sidetoside} momentarily and first complete the proof of Proposition \ref{p:o1} using it.

\begin{proof}[Proof of Proposition \ref{p:o1}]
Let $h$ be as in the statement of the proposition and fix $n\in \N$ sufficiently large.
Observe that by using Lemma \ref{l:sidetoside} and summing over all ($O(n^{2/3})$ many) $h$-compatible pairs $(I,J)$ it follows that
$$ \E \# \left\{v: [-\frac{nh}{100}, \frac{nh}{100}]^2 : \exists u\in \mathsf{Ent}_{n}, w\in \mathsf{Exit}_{n}, \mbox{slope}(u,w)\in (\frac{h}{4}, \frac{4}{h})~\text{and}~ v\in \Gamma_{u,w}\right\} \leq Cn^{5/3}$$
for some constant $C=C(h)>0$. This implies that there exists a deterministic $v\in [-\frac{nh}{100}, \frac{nh}{100}]^2 \cap \Z^2$ such that

$$\P\left( \exists u\in \mathsf{Ent}_{n}, w\in \mathsf{Exit}_{n}, \mbox{slope}(u,w)\in (\frac{h}{4}, \frac{4}{h})~\text{and}~ v\in \Gamma_{u,w}\right)\leq 10^4 h^{-2} Cn^{-1/3}.$$
For notational convenience, the event in the above display will be denoted $\widetilde{\ce}_{n,v,h}$.
Consider now $S'_{n}$: the smallest square with centre $v$ containing $S_n=[-n,n]^2$. Suppose now that the event $\ce^{*}_{n,v,h}$ holds: there exists a geodesic $\gamma$ from some $u'$ located at the left or bottom side of $S'_{n}$ to some $w'$ located at the top or right side of $S'_{n}$ such that $\mbox{slope}(u',w')\in (\frac{h}{2},\frac{2}{h})$ that passes through $v$. For $u',w'$ and $\gamma$ as above, let $u$ and $w$ denote the first and the last point where $\gamma$ intersects $S_{n}$. By the directed nature of $\gamma$ it is clear that $u\in \mathsf{Ent}_{n}$ and $w\in \mathsf{Exit}_{n}$ (notice that it is possible that $u=u'$ or $w=w'$). Let $L_{u',w'}$ denote the straight-line joining $u'$ and $w'$. It is easy to see the following: either (a) $\mbox{slope}(u,w)\in (\frac{h}{4},\frac{4}{h})$, or (b) there exist vertices on $\gamma$ whose distance to $L_{u',w'}$ exceeds $cn$ for some $c(h)>0$. Let $\mathsf{TF}^*_{n}$ denote the event that there exist points $u'$ located at the left or bottom side of $S'_{n}$ and $w'$ located at the top or right side of $S'_{n}$ such that $\mbox{slope}(u',w')\in (\frac{h}{2},\frac{2}{h})$ and such that the maximum distance of the geodesic $\Gamma_{u',w'}$ to the straightline $L_{u',w'}$ exceeds $cn$. By the above discussion, we have 
$$\ce^{*}_{n,v,h}\subseteq \widetilde{\ce}_{n,v,h} \cup \mathsf{TF}^*_{n}.$$ 
It follows from Theorem \ref{p:tf} and taking a union bound over all pairs $(u',w')$ that $\P(\mathsf{TF}^*_{n})\leq e^{-c'n}$ for some $c'(h)>0$ and for all $n$ sufficiently large. This together with the upper bound on $\P(\widetilde{\ce}_{n,v,h})$ derived above implies that $\P(\ce^{*}_{n,v,h})\leq Cn^{-1/3}$ for some $C=C(h)>0$. Notice that $S'_{n}$ is a translate of $S_{n'}$ for some $n'\in [n, n+\frac{nh}{100}]\cap \N$ and the same translation takes $\mathbf{0}$ to $v$. By translation invariance of the model $\P(\ce^{*}_{n,v,h})=\P(\ce_{n',h})$ and hence we conclude that for all $n$ sufficiently large, there exists $n'\in [n, n+\frac{nh}{100}]\cap \N$ such that $\P(\ce_{n',h})\leq Cn^{-1/3}$ completing the proof of the proposition.
%
%
%
\end{proof}

\subsection{Proof of Lemma \ref{l:sidetoside}}
We complete this section with the proof of Lemma \ref{l:sidetoside}. The key to the proof of this lemma is the following variant of Theorem \ref{t:multi2}.
Consider the lines $x+y=2nh/100$ and $x+y=-2nh/100$; denote these lines by $L_1$ and $L_2$ respectively for brevity.  
Let $(I,J)$ denote an $h$-compatible pair as in the statement of Lemma \ref{l:sidetoside}. Let us define an equivalence relation on the pairs $(u,v)$ with $u\in I, v\in J$. We say $(u,v)\sim (u',v')$ if $\Gamma_{u,v}$ and $\Gamma_{u',v'}$ coincide between $L_1$ and $L_2$. Let $M_{n,I,J}$ denote the number of equivalence classes. We have the following proposition.

\begin{proposition}
\label{t:multi3plus}
For each $h\in (0,1)$ sufficiently small, there exists $c=c(h)>0$ such that for all $\ell< n^{0.01}$ sufficiently large, $n\in \N$ sufficiently large and for each $h$-compatible pair $(I,J)$ we have
$$\P(M_{n,I,J}\geq \ell)\leq e^{-c\ell^{1/128}}.$$
\end{proposition}

Using Proposition \ref{t:multi3plus}, Lemma \ref{l:sidetoside} is almost immediate. 

\begin{proof}[Proof of Lemma \ref{l:sidetoside}]
Fix $h\in (0,1)$ sufficiently small and an $h$ compatible pair $(I,J)$. One deterministically has that $N_{n}(I,J)\leq \frac{4nh}{100}M_{n,I,J}$, and therefore it suffices to show that $\E M_{n,I,J} \leq C$ for some $C=C(h)>0$. Observing that, deterministically, $M_{n,I,J}\leq n^{4/3}$ it follows that 
$$\E M_{n,I,J} \leq \sum_{\ell< n^{0.01}} \P(M_{n,I,J}\geq \ell) +n^{4/3}\P(M_{n,I,J}\geq 0.5n^{0.01}).$$
The proof is completed using Proposition \ref{t:multi3plus} to bound both the terms above. 
\end{proof} 

It remains to provide the proof of Proposition \ref{t:multi3plus}. Notice that there are four cases two consider: (i) both $I$ and $J$ are parallel vertical line segments, (ii) both $I$ and $J$ are parallel horizontal line segments, (iii) $I$ is a vertical line segment and $J$ is a horizontal line segment, (iv) $I$ is a horizontal line segment and $J$ is a vertical line segment. We shall provide the proof for (i) and (iii); the proof for (ii) is identical to that of (i) and the proof for (iv) is identical to that of (iii) using the invariance of the LPP model under reflection on the line $x+y=0$.

\begin{figure}
\includegraphics[width=0.8\textwidth]{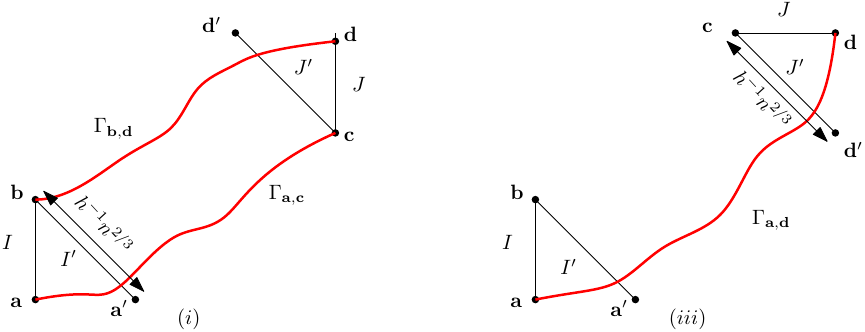}
\caption{Proof of Proposition 4.2: to be able to use Theorem \ref{t:multi2} we consider the geodesics between $I'$ and $J'$ (instead of $I$ and $J$) which are parallel to the anti-diagonal line $x+y=0$. In both scenario (i) and scenario (iii) $I'$ and $J'$ are constructed in such a way, that with high probability any geodesic from a point in $I$ to a point in $J$ will cross both $I'$ and $J'$. Indeed, by planar ordering, any geodesic from a point in $I$ to a point in $J$ is sandwiched between $\Gamma_{\mathbf{a},\mathbf{d}}$ and $\Gamma_{\mathbf{b},\mathbf{c}}$ in scenario (i) and $\Gamma_{\mathbf{a},\mathbf{c}}$ and $\Gamma_{\mathbf{b},\mathbf{d}}$ in scenario (iii).}
\label{f:IJ}
\end{figure}

\begin{proof}[Proof of Proposition \ref{t:multi3plus}]
For $(I,J)$ in cases (i) or (iii) as described above, let us consider $I'$ and $J'$ as follows (See Figure \ref{f:IJ}). In scenario (i), let us assume without loss of generality that $I$ is the interval joining the points $\mathbf{a}$ and $\mathbf{b}=\mathbf{a}+(0,n^{2/3})$, and $J$ is the interval joining $\mathbf{c}$ and  $\mathbf{d}=\mathbf{c}+(0,n^{2/3})$. Let us now denote $I'$ to be the line segment parallel to $x+y=0$ joining $\mathbf{b}$ and $\mathbf{a'}:=\mathbf{b}+100h^{-1}(n^{2/3},-n^{2/3})$. Similarly let us define $J'$ to be the line segment parallel to $x+y=0$ joining $\mathbf{c}$ and $\mathbf{d'}:=\mathbf{c}-100h^{-1}(n^{2/3},-n^{2/3})$. In scenario (iii), let $I$ and $I'$ be as above (with possibly different value of $\mathbf{a}$), let $J$ be the interval joining $\mathbf{c}$ and  $\mathbf{d}=\mathbf{c}+(n^{2/3},0)$ (again, the value of $\mathbf{c}$ might be different), and let us define $J'$ to be the line segment parallel to $x+y=0$ joining $\mathbf{c}$ and $\mathbf{d'}:=\mathbf{c}+100h^{-1}(n^{2/3},-n^{2/3})$. 

Let us now define a natural analogue of the equivalence relation $\sim$ on the pairs $(u,v)$ with $u\in I', v\in J'$. We say $(u,v)\sim (u',v')$ if $\Gamma_{u,v}$ and $\Gamma_{u',v'}$ coincide between $L_1$ and $L_2$. Let $M'_{n,I,J}$ denote the number of equivalence classes. Let $\mathsf{T}_*$ denote the event that there exists $u\in I$, $v\in J$ such that the geodesic $\Gamma_{u,v}$ does not intersect $I'$ or $J'$. Clearly,
$$\P(M_{n,I,J}\geq \ell) \leq \P(M'_{n,I,J}\geq \ell)+\P(\mathsf{T}_{*}).$$
Clearly, Proposition \ref{t:multi2} applies to $M'_{n,I,J}$ for each $h$-compatible $(I,J)$ (by choosing $\psi$ sufficiently close to $1$ depending on $h$). Also, observe that $\P(\mathsf{T}_{*})$ can be upper bounded using the planar ordering of the geodesics as follows. In scenario (i): it is upper bounded by sum of the probability that the geodesics from $\mathbf{b}$ to $\mathbf{d}$ does not intersect $J'$ and the probability that the geodesic from $\mathbf{a}$ to $\mathbf{c}$ does not intersects $I'$.  In scenario (iii), $\P(\mathsf{T}_*)$ is upper bounded by the probability that geodesic from $\mathbf{a}$ to $\mathbf{d}$ does not intersect at least one of $I'$ and $J'$. Notice that in each of these cases, $\mathsf{T}_{*}$ implies a large transversal fluctuation of one of the geodesics. It therefore follows using \cite[Theorem 3, Corollary 2.4]{BSS17++}  that $\P(\mathsf{T}_{*})\leq e^{-cn^{4/9}}\leq e^{-c\ell^{1/128}}$ for some $c>0$ and all $n$ sufficiently large where the final inequality follows from our assumption that $\ell<n^{0.01}$. The proof of the proposition is completed putting together the above with the upper bound on $\P(M'_{n,I,J}\geq \ell)$ which in turn is obtained using Theorem \ref{t:multi2}. 
\end{proof}

\section{Ruling out bigeodesics in axial directions}
\label{s:axial}
We prove Proposition \ref{p:axial} in this section using Theorem \ref{l:ltf} and provide proofs of Theorem \ref{p:tfbasic} and Theorem \ref{l:ltf}. Recall that Proposition \ref{p:axial} asserts that non-trivial bigeodesics passing through $\mathbf{0}$ directed in axial directions almost surely do not exist. By the obvious symmetry of the problem, it suffices  to rule out vertically directed bigeodesics only. 

To distinguish between two types of vertically directed geodesics, we first make the following definition. A bigeodesic $\gamma=\{v_i\}_{i\in \Z}$ with $v_0=\mathbf{0}$ with forward and backward limiting direction both equal to $\infty$ is called a \textbf{finite width bigeodesic} if $\lim_{n\to \infty} x_{n}-x_{-n} <\infty$ where $v_{n}=(x_n,y_n)$ for $n\in \Z$. A vertically upward directed semi-infinite geodesic $\gamma=\{v_i\}_{i\in \Z_{+}}$ started from $\mathbf{0}$ is called an \textbf{infinite width geodesic} if $x_{n}\to \infty$ as $n\to \infty$ where $v_n=(x_n,y_n)$. For a semi-infinite geodesic $\Gamma$ starting from $\mathbf{0}$ directed vertically upwards and $M\in \N$ we denote $\Gamma_{M}$ to be the smallest positive integer such that $(M,\Gamma_{M})\in \Gamma$.

We need the following two lemmas. 

\begin{lemma}
\label{l:fw}
Almost surely there does not exist any non-trivial vertically directed finite width bigeodesic.
\end{lemma}


\begin{lemma}
\label{l:ifw}
For a fixed $\delta>0$, there exists $M_0(\delta)$ sufficiently large such that for $M\in \N$ with $M\ge M_0$ sufficiently large and all $L\in \N$, the probability that there exists an infinite width semi-infinite geodesic $\Gamma$ started from $\mathbf{0}$ directed vertically upwards with $\Gamma_M \leq L$ is at most $\delta$.
%
%
\end{lemma}

Postponing momentarily the proofs of Lemma \ref{l:fw} and Lemma \ref{l:ifw}, we first complete the proof of Proposition \ref{p:axial}. 

\begin{proof}[Proof of Proposition \ref{p:axial}]
Observe that the set of all infinite width semi-infinite geodesics started from $\mathbf{0}$ directed vertically upwards is ordered if it is nonempty, let $\Gamma$ denote the leftmost such geodesic, if it exists. By definition, $\Gamma_{M}<\infty$ for each $M\in \N$. Hence, if $\Gamma$ exists with positive probability; for each $M$ in $\N$ there must exist $L\in \N$ such that with uniformly positive probability (say with probability at least $\theta>0$) we have $\Gamma_{M}\leq L$. By choosing $\delta$ small compared to $\theta$, and $M$ sufficiently large, this is contradicted by Lemma \ref{l:ifw} and hence almost surely there does not exist any infinite width semi-infinite geodesic started from $\mathbf{0}$ directed vertically upwards, and the same argument can be used to rule our vertically downward directed infinite width semi-infinite geodesic rays started at $\mathbf{0}$. This, together with Lemma \ref{l:fw} implies that almost surely there does not exist any non-trivial vertically directed bigeodesic passing through $\mathbf{0}$. The proof of the proposition is completed using the obvious symmetry of the problem to handle horizontally directed non-trivial bigeodesics.
\end{proof}

%

We now provide the proofs of Lemma \ref{l:fw} and Lemma \ref{l:ifw}.

\begin{proof}[Proof of Lemma \ref{l:fw}]
For $a<b\in \Z$, we call a vertically directed finite width bigeodesic $\gamma=\{v_i\}_{i\in \Z}$ (not necessarily passing through $\mathbf{0}$) an $(a,b)$ bigeodesic if $\lim_{n\to -\infty} x_{n}=a$ and $\lim_{n\to \infty}x_{n}=b$. Clearly it suffices to show that, for each $(a,b)\in \Z^2$, $a<b$; almost surely there does not exist any $(a,b)$ bigeodesic. Fix $a<b\in \Z$. For $i\in \Z$; let $\mathcal{C}(a,b;i)$ denote the event that there exists an $(a,b)$ bigeodesic $\gamma$ such that $(a,i)\in \gamma$ and $(a+1,i)\in \gamma$. Clearly, by translation invariance $\P(\mathcal{C}(a,b;i))$ does not depend on $i$. Observe also that for almost every given realization of vertex weights; there can be at most one $(a,b)$ bigeodesic and hence $\mathcal{C}(a,b;i)$ can hold for at most one $i$. This implies that $\P(\mathcal{C}(a,b;i))=0$ for all $i\in \Z$ which, in turn, implies that almost surely there does not exist any $(a,b)$ bigeodesic. Taking an union bound over all pairs $(a,b)$, we get the result.
\end{proof}


%


\begin{figure}[htbp!]
\centering
\includegraphics[width=0.25\textwidth]{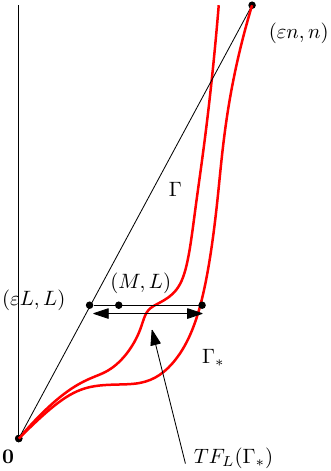}
\caption{An illustration of the argument in the proof of Lemma \ref{l:ifw}. By way of contradiction we assume that for a vertically directed semi-infinite geodesic started from the origin, we have $\Gamma_{M}\leq L$ with probability bounded below. This ensures, by definition, that $\Gamma$ passes to the right of the point $(M,L)$. Now we choose $\varepsilon$ sufficiently small so the point $(\varepsilon L, L)$ lies to the left of $\Gamma$ and for $n$ sufficiently large with uniformly positive probability the geodesic $\Gamma_{*}$ from $\mathbf{0}$ to $(\varepsilon n,n)$ lies to the right of $\Gamma$. This implies that $TF_{L}(\Gamma_*)$, the local transversal fluctuation of $\Gamma_{*}$ at scale $L$ is much larger than typical and a contradiction is obtained using Theorem \ref{l:ltf}.}
\label{f:steep2}
\end{figure}

\begin{proof}[Proof of Lemma \ref{l:ifw}]
We shall prove Lemma \ref{l:ifw} by contradiction. Let $\delta>0$ be fixed. Let $c$ and $C_0$ be as in Theorem \ref{l:ltf}. Let us fix $M\geq 10C_0\vee M_{0}(\delta)$ where $M_0$ is to be chosen sufficiently large later. By way of contradiction let us suppose that $L\in \N$ is such that with probability more than $\delta$ there exists a semi-infinite geodesic $\Gamma$ as in the statement of the lemma with $\Gamma_{M}\leq L$. Observe that we can take $L$ to be sufficiently large ($L>L_0$ where $L_0$ is as in Theorem \ref{l:ltf}). Let us now fix $\varepsilon\geq C_0/L$ such that $M\geq 2\varepsilon L$. 

By the assumed vertical direction of $\Gamma$, there must exist (random) $n$ sufficiently large such that the point $(\varepsilon n, n)$, lies to the right of $\Gamma$ (see Figure \ref{f:steep2}), and by the hypothesis and planar ordering of the geodesics there exists (deterministic) $n$ sufficiently large such that with probability at least $\delta/2$, the geodesic $\Gamma_{*}$ from $\mathbf{0}$ to $(\varepsilon n, n)$ lies to the right of $\Gamma$.
Using the notation of Theorem \ref{l:ltf} this implies that $TF_{L}(\Gamma_{*})\geq M-\varepsilon L \geq M/2$ where in the last inequality we use $M\geq 2\varepsilon L$, and hence $\P(TF_{L}(\Gamma_{*})\geq M/2)\geq \delta/2$.

Let us now choose $M_0$ sufficiently large so that $(M/2)>x_0^{3}$ where $x_0$ is as in Theorem \ref{l:ltf} (this ensures that $M/(2(\varepsilon L)^{2/3})\geq (M/2)^{1/3}\geq  x_0$ and hence Theorem \ref{l:ltf} applies) and $e^{-c(M/2)^{1/9}}<\delta/2$. These two assumptions imply, invoking Theorem \ref{l:ltf}, that $\P(TF_{L}(\Gamma_{*})\geq M/2)\leq e^{-c(M/2)^{1/9}}<\delta/2$, which is a contradiction. This completes the proof.
%
%
%
\end{proof}

\subsection{Transversal Fluctuation of Steep Geodesics}
It remains to prove Theorem \ref{p:tfbasic} and Theorem \ref{l:ltf} which is done in this subsection. Recall that for  any path $\gamma$ from $\mathbf{0}$ to $(\varepsilon n, n)$, let the local transversal fluctuation of $\gamma$ at length scale $L$ be
$$TF_{L}(\gamma):=\sup\{(x-\varepsilon L)_{+}: (x,L)\in \gamma\}.$$ 
Theorem \ref{p:tfbasic} asserts that for $m\in (\frac{\varepsilon}{10},10\varepsilon)$, and the geodesic $\Gamma$ from $\mathbf{0}$ to $(mn,n)$, we have with large probability $TF_{L}(\Gamma)=O(\varepsilon^{2/3}n^{2/3})$ for each $L\in [\frac{n}{4},\frac{n}{2}]$. We first need an auxiliary result which is an analogue of Theorem \ref{t:supinf}, (ii). 
For $m, \varepsilon>0$, $n\in \N$, let $U=U_{n,m,\varepsilon}$ denote the parallelogram whose vertices are $\mathbf{0}$, $(\varepsilon^{2/3}n^{2/3},0)$, $(mn,n)$, $(mn+\varepsilon^{2/3}n^{2/3},n)$. Let $A_{U}$ and $B_{U}$ denote the bottom and top side of $U$ respectively. We have the following proposition. 

\begin{proposition}
\label{p:supsteep}
There exist constants $C_0,c,x_0, \varepsilon_0>0$ such that for each $\varepsilon\in (\frac{C_0}{n}, \varepsilon_0)$, $m\in (\frac{\varepsilon}{100}, 100 \varepsilon)$, $n$ sufficiently large and $x>x_0$, we have the following:
$$\P\left(\sup_{u\in A_{U}, v\in B_{U}} T_{u,v}-\E T_{u,v} \geq x \varepsilon^{-1/6} n^{1/3}\right)\leq e^{-cx}.$$
\end{proposition}

\begin{proof}
The argument presented here is similar to \cite[Proposition 3.5]{BGHH20}. Let $G$ denote the event that 
$$\sup_{u\in A_{U}, v\in B_{U}} T_{u,v}-\E T_{u,v}\geq x \varepsilon^{-1/6} n^{1/3};$$
and let $G_{u,v}$ denote the event that in addition the supremum above is attained for $u\in A_{U}, v\in {B}_{U}$. Clearly $G_{u,v}$ are disjoint and $\cup_{u,v} G_{u,v}= G$. Consider the straight-line joining the midpoints of $A_{U}$ and $B_{U}$ and let $u_0$ and $v_0$ be the points where this line intersects the lines $y=-n$ and $y=2n$ respectively.  

Notice that \eqref{e:meansteep} implies that 
$$\sup_{u\in A_{U},v\in B_{U}}|\E T_{u_0,v_0}-\E T_{u_0,u}- \E T_{u,v}-\E T_{v,v_0}|\leq C'\varepsilon^{-1/6}n^{1/3}$$
for some $C'>0$ (here we have used the hypothesis that $m/\varepsilon$ is bounded above and below). For $x$ sufficiently large it now follows that 
$$\P\left(T_{u_0,v_0}-\E T_{u_0,v_0}\geq \frac{x}{2} \varepsilon^{-1/6}n^{1/3}\right)\geq \sum_{u\in A_{U},v\in B_{U}} \P(G_{u,v}\cap H^1_{u,v}\cap H^2_{u,v})$$
where $H^1_{u,v}$ (resp.\ $H^2_{u,v}$) denotes the event $T_{u_0,u}-\E T_{u_0,u} \geq -\frac{x}{10}\varepsilon^{-1/6}n^{1/3}$ (resp.\ $T_{v,v_0}-\E T_{v,v_0} \geq -\frac{x}{10}\varepsilon^{-1/6}n^{1/3}$). Notice that $G_{u,v}, H^1_{u,v}$ and $H^2_{u,v}$ are independent and \eqref{e:steep2} (and \eqref{e:meansteep}) implies that for $x$ sufficiently large $\P(H^1_{u,v}), \P(H^2_{u,v})\geq \frac{1}{2}$ for all $u\in A_{U}, v\in B_{U}$. It follows that 
$$\P\left(T_{u_0,v_0}-\E T_{u_0,v_0}\geq \frac{x}{2} \varepsilon^{-1/6}n^{1/3}\right)\geq \frac{1}{4}\sum_{u\in A_{U},v\in B_{U}} \P(G_{u,v})=\frac{1}{4}\P(G).$$
Using \eqref{e:steep} to conclude that the left hand side of the above display is upper bounded by $e^{-cx}$ completes the proof of the proposition.
\end{proof}

%


We are now ready to prove Theorem \ref{p:tfbasic}.

\begin{proof}[Proof of Theorem \ref{p:tfbasic}]
For the sake of brevity, we shall write a detailed proof only in the special case of $L=\frac{n}{2}$, the same argument goes through for any $L\in [\frac{n}{4},\frac{n}{2}]$ with little to no change as indicated at the end of this proof. The proof for the case $L=\frac{n}{2}$ is similar to the proof of \cite[Lemma 11.3]{BSS14}. Fix $x>0$ sufficiently large. Clearly if $x> \frac{1}{2}m\varepsilon^{-2/3}n^{1/3}$ there is nothing to prove, so let us assume that  $x\leq \frac{1}{2}m\varepsilon^{-2/3} n^{1/3}$. For $j\geq 0$, let $A_{j}$ denote the line segment joining $(mn/2+(x+j)\varepsilon^{2/3}n^{2/3},n/2)$ and $(mn/2+(x+j+1)\varepsilon^{2/3}n^{2/3},n/2)$.  Let $\mathcal{A}_{j}$ denote the event that 
$$\sup_{v\in A_j} T_{\mathbf{0},v}+T_{v, (mn,n)} \geq n(1+\sqrt{m})^2 -x\varepsilon^{-1/6}n^{1/3}.$$ Let $\mathcal{B}$ denote the event that $$T_{\mathbf{0}, (mn,n)} \leq n(1+\sqrt{m})^2 -x\varepsilon^{-1/6}n^{1/3}.$$ 
Finally let $\mathcal{C}$ denote the event that 
$$T_{\mathbf{0}, (mn, n/2)}+ T_{(0.9mn,n/2), (mn,n)} \geq n(1+\sqrt{m})^2 -x\varepsilon^{-1/6}n^{1/3}.$$ Clearly, 
$$\P(TF_{\frac{n}{2}}(\Gamma) \geq x\varepsilon^{2/3}n^{2/3})\leq \P(\mathcal{C})+\P(\mathcal{B})+ \sum_{j=0}^{0.4m \varepsilon^{-2/3}n^{1/3}-x}\P(\mathcal{A}_{j}).$$
Notice that the third term in the above sum is empty if $x>0.4m\varepsilon^{-2/3}n^{1/3}$.
It follows from \eqref{e:steep2} (and that $m/\varepsilon$ is bounded above and below) that $\P(\mathcal{B})\leq e^{-cx}$ for some $c>0$. To upper bound $\P(\mathcal{C})$ notice that 
$$\frac{n}{2}\left[(1+\sqrt{2m})^2+ (1+\sqrt{0.2 m})^2\right]=n(1+\sqrt{m})^2-c'n\sqrt{m}$$
for some absolute constant $c'>0$. Notice further that if $C_0$ is sufficiently large, then $x\leq \frac{1}{2}m\varepsilon^{-2/3} n^{1/3}$, $m\in ( \varepsilon/10, 10\varepsilon)$ and $\varepsilon> C_0/n$, together imply that $c'n\sqrt{m}\ge 3x\varepsilon^{-1/6}n^{1/3}$ and it follows that
$$\P(\mathcal{C})\leq \P(T_{\mathbf{0}, (mn, n/2)}-\frac{n}{2}(1+\sqrt{2m})^2\geq x\varepsilon^{-1/6}n^{1/3})+ \P(T_{(0.9mn,n/2), (mn,n)}-\frac{n}{2}(1+\sqrt{0.2m})^2) \geq x\varepsilon^{-1/6}n^{1/3})$$
and finally \eqref{e:steep} implies that $\P(\mathcal{C})\leq e^{-cx}$ for some $c>0$. 

It remains to consider the case $x\leq 0.4m\varepsilon^{-2/3}n^{1/3}$ and bound $\P(\mathcal{A}_j)$. It follows from \eqref{e:meansteep} that for each $j$, 

\begin{equation}
\label{e:steeppen}
\sup_{v\in A_j} \E T_{\mathbf{0},v}+\E T_{v, (mn,n)} \leq n(1+\sqrt{m})^2-c'(x+j)^2\varepsilon^{-1/6}n^{1/3}
\end{equation}
for some constant $c'>0$. Arguing as above, but using Proposition \ref{p:supsteep} instead, it now follows that for $x$ sufficiently large $\P(\mathcal{A}_j)\leq e^{-c(x+j)}$ for some $c>0$. Summing over all $j$, and taking a union bound over $\cup_{j}\mathcal{A}_j, \mathcal{B}$ and $\mathcal{C}$, this completes the proof of the proposition for the case $L=\frac{n}{2}$. 

For any $L\in [\frac{n}{4},\frac{n}{2}]$, one can define $A_j$ to be the line segment joining $(mL+(x+j)\varepsilon^{2/3}n^{2/3},L)$ and $(mL+(x+j+1)\varepsilon^{2/3}n^{2/3},L)$, and define $\mathcal{A}_{j}$ similarly as before. Let $\mathcal{B}$ be the same as above, and let $\mathcal{C}$ be defined by changing $\frac{n}{2}$ to $L$ in an obvious way. Since $\frac{L}{n}$ is bounded away from $0$ and $1/2$, \eqref{e:steeppen} continues to hold (with possibly a changed value of $c'$) and one can bound the probabilities of $\mathcal{A}_j$, $\mathcal{B}$ and $\mathcal{C}$ in the same manner as above to conclude the proof for any $L\in [\frac{n}{4},\frac{n}{2}]$.
\end{proof}

We can now complete the proof of Theorem \ref{l:ltf}; this is similar to \cite[Theorem 3]{BSS17++}. 

\begin{figure}[h]
\centering
\includegraphics[width=0.4\textwidth]{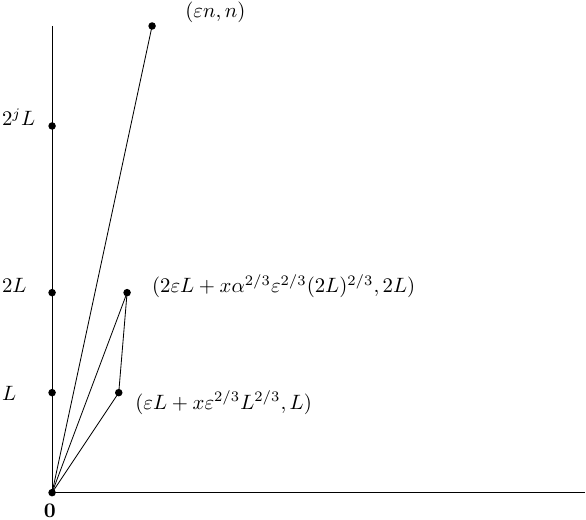}
\caption{Theorem \ref{l:ltf} shows that it is unlikely that the geodesic from $\mathbf{0}$ to $(\varepsilon n, n)$ has a large transversal fluctuation (at the scale $\varepsilon^{2/3}L^{2/3}$) at height $L$. To prove this one shows that it is unlikely that the best path from $\mathbf{0}$ to $(2\varepsilon L+x\alpha^{2/3}\varepsilon^{2/3}(2L)^{2/3}, 2L)$ via $(\varepsilon L+x\varepsilon^{2/3}L^{2/3}, L)$ is competitive with the geodesic from $\mathbf{0}$ to $(2\varepsilon L+x\alpha^{2/3}\varepsilon^{2/3}(2L)^{2/3}, 2L)$. Doing this calculation at all dyadic scales and summing over scales gives the desired result.}
\label{f:steep}
\end{figure}

\begin{proof}[Proof of Theorem \ref{l:ltf}]
{Without loss of generality let us for now} assume that $n=2^{j_0}L$ for some $j_0\in \N$. Fix $x$ sufficiently large. Fix a real number $\alpha\in (1,\sqrt{2})$. For $j\leq j_0$, let $\mathcal{A}_{j}$ denote the event that $TF_{2^{j}L}(\Gamma) \geq x(\alpha^{j}\varepsilon 2^{j}L)^{2/3}$. Clearly it suffices to show that 
$$\sum_{j\geq 1} \P(\mathcal{A}_{j}^c\cap \mathcal{A}_{j-1}) \leq e^{-cx^{1/3}}.$$ 

Observe that, by planar ordering of the geodesics, one has, on $\mathcal{A}^c_{j}\cap \mathcal{A}_{j-1}$, the geodesic $\Gamma^{*}$ from $\mathbf{0}$ to $v':=(\varepsilon 2^{j}L+x(\alpha^{j}\varepsilon 2^{j}L)^{2/3}), 2^{j}L)$ lies to the right of $\Gamma$ and hence it intersects the line $y=2^{j-1}L$ at some point to the right of $(\varepsilon 2^{j-1}L+x(\alpha^{j-1}\varepsilon 2^{j-1}L)^{2/3}), 2^{j-1}L)$. It follows from definition that 

$$TF_{2^{j-1}L}(\Gamma^{*}) \geq x(\alpha^{j-1}\varepsilon 2^{j-1}L)^{2/3}-\frac{x}{2}(\alpha^{j}\varepsilon 2^{j}L)^{2/3}>0$$
using $\alpha<\sqrt{2}$. By choosing $\alpha$ appropriately, we ensure that $\frac{1}{(2\alpha)^{2/3}}-\frac{1}{2}\geq 0.1$ and hence 
$$TF_{2^{j-1}L}(\Gamma^{*})\geq 0.1 x\alpha^{2j/3} (\varepsilon 2^{j}L)^{2/3}.$$
%
Now we need to consider two cases.
 
\textbf{Case 1:} If $x\alpha^{2j/3} (\varepsilon 2^{j}L)^{2/3} \leq  9\varepsilon 2^{j}L$, then $v'=(m2^{j}L,2^{j}L)$ for some $m\in (\frac{\varepsilon}{10}, 10\varepsilon)$. Since $\varepsilon L$ is assumed to be sufficiently large, Proposition \ref{p:tfbasic} applies to $\Gamma^{*}$ and we get that for some $c>0$
$$\P(\mathcal{A}_{j}\cap \mathcal{A}_{j-1}^{c})\leq e^{-cx\alpha^{2j/3}}\leq e^{-cx^{1/3}\alpha^{2j/9}}.$$
 
   
\textbf{Case 2:} In the other case, notice that we cannot directly apply Proposition \ref{p:tfbasic} to $\Gamma^*$ with the same $\varepsilon$. Consider $\varepsilon'=(2^{j}L)^{-1}x(\alpha^{j}\varepsilon 2^{j}L)^{2/3}$. Notice that in this case 
$v'=((\varepsilon+\varepsilon')2^{j}L, 2^jL)$, and by the assumption $x\alpha^{2j/3} (\varepsilon 2^{j}L)^{2/3} >  9\varepsilon 2^{j}L$, we have $\varepsilon+\varepsilon'<10\varepsilon'$. Observe also that in this case we have
$$TF_{2^{j-1}L}(\Gamma^*)\geq 0.1 \varepsilon' 2^{j}L= 0.1 (\varepsilon' 2^{j}L)^{1/3} (\varepsilon' 2^{j}L)^{2/3}.$$
Finally notice that $\varepsilon' 2^j L \geq 9\varepsilon 2^{j}L$ and hence is sufficiently large (recall $\varepsilon L$ is assumed to be sufficiently large) and hence Proposition \ref{p:tfbasic} applies to $\Gamma^*$ with $\varepsilon$ replaced by $\varepsilon'$ and hence we have 
$$\P(\mathcal{A}_{j}\cap \mathcal{A}_{j-1}^{c}) \leq e^{-c'(\varepsilon' 2^{j}L)^{1/3}}$$
for some $c>0$. Now, $\varepsilon' 2^{j}L=x\alpha^{2j/3}(\varepsilon 2^{j}L)^{2/3}$ and hence in this case we get $\P(\mathcal{A}_{j}\cap \mathcal{A}_{j-1}^{c})\leq e^{-cx^{1/3}\alpha^{2j/9}}$ for some $c>0$ in this case also. 

Combining the two cases above and summing this over all $j$ gives the result when $\frac{n}{L}$ is an integer power of $2$. Suppose now $n\in (2^{j_0}L, 2^{j_0+1}L)$ for some $j_0\ge 3$ (recall we have assumed $L\leq n/8$). Clearly we can run through the same argument as above provided we can show that 
$$\P\left(TF_{2^{j_0-1}L}(\Gamma) \geq x(\alpha^{j_0-1}\varepsilon 2^{j_0-1}L)^{2/3}\right)\leq e^{-cx^{1/3}}$$ 
for some $c>0$. Since $2^{j_0-1}L\in [\frac{n}{4},\frac{n}{2}]$, and $x$ is sufficiently large, this follows from Theorem \ref{p:tfbasic} and we are done in the more general case as well. 
%
\end{proof}

\bibliography{slowbond}
\bibliographystyle{plain}

\Addresses
\end{document}